\documentclass[11pt]{article}
\usepackage{amssymb,amsmath,epsfig}
\usepackage[colorlinks=false,breaklinks=true,linkcolor=blue]{hyperref}
\usepackage{graphics,psfrag,graphicx,color}
\usepackage{diagbox}
\usepackage{float,hhline}

\numberwithin{equation}{section}
\newcommand{\ds}{\displaystyle}

\newcommand{\bsi}{{\boldsymbol\sigma}}

\newcommand{\btau}{{\boldsymbol\tau}}

\newcommand{\bchi}{{\boldsymbol\chi}}

\newcommand\cI{{\mathcal I}}
\newcommand\bcI{{\boldsymbol\cI}}

\newcommand{\bv}{{\mathbf{v}}}
\newcommand{\bw}{{\mathbf{w}}}
\newcommand{\f}{\mathbf{f}}
\newcommand{\g}{\mathbf{g}}
\newcommand{\ba}{\mathbf{a}}
\newcommand{\bb}{\mathbf{b}}

\newcommand{\bi}{\mathbf{i}}

\newcommand{\bu}{\mathbf{u}}

\newcommand{\bz}{{\mathbf{z}}}
\newcommand{\bt}{{\mathbf{t}}}
\newcommand{\bn}{{\mathbf{n}}}

\newcommand{\0}{{\mathbf{0}}}

\def\bC{\mathbf{C}}

\def\bK{\mathbf{K}}

\def\bX{\mathbf{X}}
\def\bV{\mathbf{V}}

\def\bQ{\mathbf{Q}}

\def\bRT{\mathbf{RT}}
\def\bBR{\mathbf{BR}}

\def\bx{\mathbf{x}}
\def\by{\mathbf{y}}
\def\bz{\mathbf{z}}
\newcommand{\bL}{\mathbf{L}}
\newcommand\bH{\mathbf{H}}

\newcommand\bbN{\mathbb{N}}

\newcommand\bbI{\mathbb{I}}

\newcommand\bbH{\mathbb{H}}

\newcommand\bbL{\mathbb{L}}

\newcommand{\bbC}{\mathbb{C}}

\newcommand{\cC}{\mathcal{C}}

\newcommand{\cE}{\mathcal{E}}

\newcommand{\cT}{\mathcal{T}}

\newcommand{\cO}{\mathcal{O}}
\newcommand{\cP}{\mathcal{P}}

\newcommand{\cR}{\mathcal{R}}
\def\R{\mathrm{R}}

\def\H{\mathrm{H}}

\def\L{\mathrm{L}}

\def\X{\mathrm{X}}
\def\W{\mathrm{W}}
\def\re{\mathsf{e}}
\def\sr{\mathsf{r}}

\def\re{\mathsf{e}}
\def\sr{\mathsf{r}}
\def\rB{\mathrm{B}}
\def\rD{\mathrm{D}}

\def\rP{\mathrm{P}}

\def\rt{\mathrm{t}}

\def\tF{\mathtt{F}}

\def\tBFD{\mathtt{BFD}}

\def\bdiv{\mathbf{div}}
\def\bcurl{\mathbf{curl}}

\def\tr{\mathrm{tr}}

\def\rot{\mathrm{rot}}
\def\div{\mathrm{div}}

\def\pil{\left<}
\def\pir{\right>}

\def\coeff{\mathbf{coeff}}

\def\tol{\textsf{tol}}
\def\DOF{\mathtt{DoF}}
\def\eff{\textsf{eff}}

\def\rS{\mathrm{S}}

\def\qin{{\quad\hbox{in}\quad}}
\def\qon{{\quad\hbox{on}\quad}}
\def\qan{{\quad\hbox{and}\quad}}

\def\ov{\overline}

\def\wh{\widehat}

\newcommand{\Chel}{C_{\rm hel}}

\def\jump#1{\text{ $\hspace{-0.1cm}\left[\!\left[#1\right]\!\right]$}}
\newtheorem{thm}{Theorem}[section]

\newtheorem{lem}[thm]{Lemma}

\newenvironment{proof}{\noindent{\it Proof.}}{\hfill$\square$}

\numberwithin{equation}{section}
\numberwithin{figure}{section}
\numberwithin{table}{section}

\allowdisplaybreaks

\title{{\it A posteriori} error analysis of a mixed FEM for the coupled Brinkman--Forchheimer/Darcy problem\thanks{This research was supported by ANID-Chile through the projects {\sc Centro de Modelamiento Matem\'atico} (FB210005) and Fondecyt 11220393; by DI-UCSC through the project FGII 06/2024; and by Grupo de Investigaci\'on en An\'alisis Num\'erico y C\'alculo Cient\'ifico (GIANuC$^2$), Universidad Cat\'olica de la Sant\'isima Concepci\'on.}}

\author{{\sc Sergio Caucao}\thanks{GIANuC$^2$ and Departamento de Matem\'atica y F\'isica Aplicadas, 
Universidad Cat\'olica de la Sant\'isima Concepci\'on, Casilla 297, Concepci\'on, Chile, 
email: {\tt scaucao@ucsc.cl}.}
\quad
{\sc Paulo Z{\'u}{\~n}iga}\thanks{Departamento de Ciencias Matem\'aticas y F\'isicas, Universidad Cat\'olica de Temuco, Temuco, Chile, email: {\tt paulo.zuniga@uct.cl}.}}

\date{ }

\begin{document}

\maketitle

\begin{abstract}
\noindent We consider a mixed variational formulation recently proposed for the coupling of the Brinkman--Forchheimer and Darcy equations and develop the first reliable and efficient residual-based {\it a posteriori} error estimator for the 2D version of the associated conforming mixed finite element scheme. For the reliability analysis, due to the nonlinear nature of the problem, we make use of the inf-sup condition and the strong monotonicity of the operators involved, along with a stable Helmholtz decomposition in Hilbert spaces and local approximation properties of the Raviart--Thomas and Cl\'ement interpolants. On the other hand, inverse inequalities, the localization technique through bubble functions, and known results from previous works are the main tools yielding the efficiency estimate. Finally, several numerical examples confirming the theoretical properties of the estimator and illustrating the performance of the associated adaptive algorithms are reported. In particular, the case of flow through a heterogeneous porous medium is considered.
\end{abstract}



\maketitle


\section{Introduction}

We recently introduced in \cite{cd2023} a mixed finite element method for modeling the filtration of an incompressible fluid through a non-deformable saturated porous medium with heterogeneous permeability. The flows are governed by the Brinkman--Forchheimer and Darcy equations in the more and less permeable regions, respectively, with the corresponding transmission conditions defined by mass conservation and continuity of momentum. The standard mixed formulation is used in the Brinkman--Forchheimer domain and the dual-mixed formulation in the Darcy region, with the continuity of normal velocities enforced through the introduction of a suitable Lagrange multiplier. The finite element discretization employs Bernardi--Raugel and Raviart--Thomas elements for the velocities, piecewise constants for the pressures, and continuous piecewise linear elements for the Lagrange multiplier. Stability, convergence, and optimal {\it a priori} error estimates were also derived in \cite{cd2023}.

It is widely recognized that adaptive algorithms using {\it a posteriori} error estimates are highly effective in recovering the loss of convergence orders in most standard Galerkin procedures, such as finite element and mixed finite element methods. This is especially true when these methods are applied to nonlinear problems with singularities or high gradients in the exact solutions.
In particular, this powerful tool has been applied to quasi-Newtonian fluid flows obeying the power law, which includes Forchheimer and related single and coupled models.
In this direction, we refer to \cite{ep2006}, \cite{fz2008}, \cite{s2021}, \cite{cgos2021}, \cite{ce2024}, and \cite{cgg2024-pp}, for various contributions addressing this issue.
Particularly, in \cite{ep2006} an {\it a posteriori} error estimator defined via a non-linear projection of the residuals of the variational equations for a three-field model of a generalized Stokes problem was proposed and analyzed.
In turn, a new {\it a posteriori} error estimator for a mixed finite element approximation of non-Newtonian fluid flow problems was developed in \cite{fz2008}. This mixed formulation, like finite volume methods, possesses local conservation properties, namely conservation of momentum and mass.
Later on, {\it a posteriori} error analyses for the aforementioned Brinkman--Darcy--Forchheimer model in velocity-pressure formulation were developed in \cite{s2021}. Specifically, two types of error indicators related to the discretization and linearization of the problem were established.
Furthermore, the first contribution to deriving an {\it a posteriori} error analysis of the primal-mixed finite element method for the Navier--Stokes/Darcy--Forchheimer coupled problem was proposed and analyzed in \cite{cgos2021}. Specifically, this work extended the usual techniques from the Hilbertian framework to Banach spaces, resulting in a reliable and efficient {\it a posteriori} error estimator for the mixed finite element method. The analysis includes corresponding local estimates and introduces new Helmholtz decompositions for reliability, as well as inverse inequalities and local estimates of bubble functions for efficiency.
Meanwhile, \cite{ce2024} presents the first {\it a posteriori} error analysis for an augmented mixed finite element method applied to the stationary convective Brinkman--Forchheimer equations within a Hilbert framework. The latter was also studied in \cite{cgg2024-pp} for non-augmented mixed finite element methods based on Banach spaces in a Banach framework.
Finally, we refer to \cite{cgoz2022} and \cite{cgo2023} for recent {\it a posteriori} error analyses of partially augmented and Banach spaces-based mixed formulations for the coupled Brinkman--Forchheimer and double-diffusion equations.

According to the above discussion, and to complement the study started in \cite{cd2023} for the coupled Brinkman--Forchheimer/Darcy problem, in the present paper we employ and adapt the {\it a posteriori} error analysis techniques developed in \cite{bg2010}, \cite{agr2016}, and \cite{ce2024} for mixed formulations to the current coupled problem. We develop a reliable and efficient residual-based {\it a posteriori} error estimator in 2D for the mixed finite element method from \cite{cd2023}. More precisely, we derive a global quantity $\Theta$ that is expressed in terms of calculable local indicators $\Theta_T$ defined on each element $T$ of a given triangulation $\mathcal{T}$. This information can then be used to localize sources of error and construct an algorithm to efficiently adapt the mesh. In this way, the estimator $\Theta$ is said to be efficient (resp. reliable) if there exists a positive constant $C_{\mathtt{eff}}$ (resp. $C_{\mathtt{rel}}$), independent of the mesh sizes, such that
\[
C_{\mathtt{eff}}\,\Theta \,+\, \mathtt{h.o.t.} \,\,\le\,\, \|\mathrm{error}\|\,\,\le\,\,
C_{\mathtt{rel}}\,\Theta \,+\, \mathtt{h.o.t.}\,,
\]
where {\tt h.o.t.} is a generic expression denoting one or several terms of higher order. We remark that up to the authors' knowledge, the present work provides the first {\it a posteriori} error analyses of mixed finite element methods for the coupled Brinkman--Forchheimer/Darcy equations.

This paper is organized as follows. The remainder of this section introduces some 
standard notations and functional spaces. In Section \ref{sec:model-and-vf}, 
we recall from \cite{cd2023} the model problem and its continuous and 
discrete mixed variational formulations. Next, in Section \ref{sec:a-posteriori-error-analysis}, 
we derive in full detail a reliable and efficient residual-based {\it a posteriori} 
error estimator in 2D. Several numerical results illustrating the reliability and efficiency 
of the {\it a posteriori} error estimator, as well as the good performance of 
the associated adaptive algorithm and the recovery of optimal rates of convergence, 
are reported in Section \ref{sec:numerical-results}. Finally, further properties 
to be utilized for the derivation of the reliability and efficiency estimates are provided 
in Appendices \ref{app:preliminaries-for-reliability} and \ref{app:preliminaries-for-efficiency}, respectively.

\subsection{Preliminary notations}

Let $\cO\subset \R^2$ be a domain with Lipschitz boundary $\Gamma$. 
For $s\geq 0$ and $t\in[1,+\infty]$, we denote by $\L^t(\cO)$ and $\W^{s,t}(\cO)$ 
the usual Lebesgue and Sobolev spaces endowed with the norms $\|\cdot\|_{0,t;\cO}$ and $\|\cdot\|_{s,t;\cO}$, respectively.
Note that $\W^{0,t}(\cO)=\L^t(\cO)$. 
If $t = 2$, we write $\H^{s}(\cO)$ instead of $\W^{s,2}(\cO)$, and denote 
the corresponding norm and seminorm by $\|\cdot\|_{s,\cO}$ and $|\cdot|_{s,\cO}$, respectively. We will denote the corresponding vectorial and tensorial counterparts of a generic scalar functional space $\H$ by $\bH$ and $\bbH$. 
The $\L^2(\Gamma)$ inner product or duality pairing
is denoted by $\pil\cdot,\cdot\pir_\Gamma$.
In turn, for any vector field $\bv:=(v_1,v_2)$, we set the gradient and divergence operators as
\begin{equation*}
\nabla\bv := \left(\frac{\partial\,v_i}{\partial\,x_j}\right)_{i,j=1,2}\qan
\div(\bv) := \sum^{2}_{j=1} \frac{\partial\,v_j}{\partial\,x_j}.
\end{equation*}
When no confusion arises $| \cdot |$ will denote the Euclidean norm in $\R^2$.
In addition, in the sequel we will make use of the well-known H\"older inequality given by
\begin{equation*}
\int_{\cO} |f\,g| \leq \|f\|_{0,t;\cO}\,\|g\|_{0,t^*;\cO}
\quad \forall\, f\in \L^t(\cO),\,\forall\, g\in \L^{t^*}(\cO), 
\quad\mbox{with}\quad \frac{1}{t} + \frac{1}{t^*} = 1 \,.
\end{equation*}
Finally, we recall that $\H^1(\cO)$ is continuously embedded into $\L^t(\cO)$ for $t\geq 1$ 
in $\R^2$.
More precisely, we have the following inequality
\begin{equation}\label{eq:Sobolev-inequality}
\|w\|_{0,t;\cO} 
\,\leq\, C_{i_t} 
\|w\|_{1,\cO}\quad 
\forall\,w \in \H^1(\cO), 
\end{equation}
with $C_{i_t}>0$ a positive constant 
depending only on $|\cO|$ and $t$ (see \cite[Theorem 1.3.4]{Quarteroni-Valli}).


\section{The model problem and its variational formulation}\label{sec:model-and-vf}

In this section we recall from \cite{cd2023} the model problem, its mixed variational formulation, and the associated conforming mixed finite element method.

\subsection{The coupled Brinkman--Forchheimer/Darcy problem}

In order to describe the geometry, we let $\Omega_\rB$ and $\Omega_\rD$ be two bounded and simply connected polygonal domains in $\R^2$ such that $\partial\Omega_\rB\cap\partial\Omega_\rD = \Sigma \neq \emptyset$ and $\Omega_\rB\cap\Omega_\rD = \emptyset$. 
Then, let $\Gamma_\rB := \partial\Omega_\rB \setminus \ov{\Sigma}$, $\Gamma_\rD := \partial\Omega_\rD \setminus \ov{\Sigma}$, and denote by $\bn$ the unit normal vector on the boundaries, which is chosen pointing outward from $\Omega := \Omega_\rB\cup\Sigma\cup\Omega_\rD$ and $\Omega_\rB$ (and hence inward to $\Omega_\rD$ when seen on $\Sigma$). 
On $\Sigma$ we also consider a unit tangent vector $\bt$ (see Figure~\ref{fig:dominio-2d}). 
Then, given source terms $\f_\rB$, $\f_\rD$, and $g_\rD$, we are interested in the coupling of the Brinkman--Forchheimer and Darcy equations, which is formulated in terms of the velocity-pressure pair $(\bu_\star,p_\star)$ in $\Omega_\star$, with $\star\in \{\rB,\rD\}$.
More precisely, the sets of equations in the Brinkman--Forchheimer and Darcy domains $\Omega_\rB$ and $\Omega_\rD$, are, respectively,
\begin{equation}\label{eq:BF-model}
\begin{array}{c}
\bsi_\rB = - p_\rB\bbI + \mu\nabla\bu_\rB \qin \Omega_\rB, \quad
\bK^{-1}_\rB\bu_\rB + \tF\,|\bu_{\rB}|^{\rho-2}\bu_{\rB} - \bdiv(\bsi_\rB) = \f_\rB \qin \Omega_\rB, \\ [1ex]
\div(\bu_\rB) = 0 \qin \Omega_\rB, \quad \bu_\rB = \0 \qon \Gamma_\rB,
\end{array}
\end{equation}
and
\begin{equation}\label{eq:Darcy-model}
\bK^{-1}_\rD\bu_\rD + \nabla p_\rD = \f_\rD \qin \Omega_\rD, \quad
\div(\bu_\rD) = g_\rD \qin \Omega_\rD, \quad \bu_\rD\cdot\bn = 0 \qon \Gamma_\rD,
\end{equation}
where $\bsi_\rB$ is the Cauchy stress tensor, $\mu$ is the kinematic viscosity of the fluid,  
$\tF > 0$ is the Forchheimer coefficient, $\rho$ is a given number in $[3,4]$, and 
$\bK_{\star}\in \bbL^\infty(\Omega_\star)$ are symmetric
tensors in $\Omega_\star$, with $\star\in \{\rB,\rD\}$, equal to the symmetric permeability tensors scaled by the kinematic viscosity. 
Throughout the paper we assume that there exists $C_{\bK_\star} > 0$ such that 
\begin{equation}\label{eq:permeability-constrain}
\bw\cdot\bK^{-1}_\star(\bx)\bw \geq C_{\bK_\star} |\bw|^2,
\end{equation}
for almost all $\bx\in\Omega_\star$, and for all $\bw\in\R^2$.
%
%
%
\begin{figure}[H]
\centering\includegraphics[scale=1]{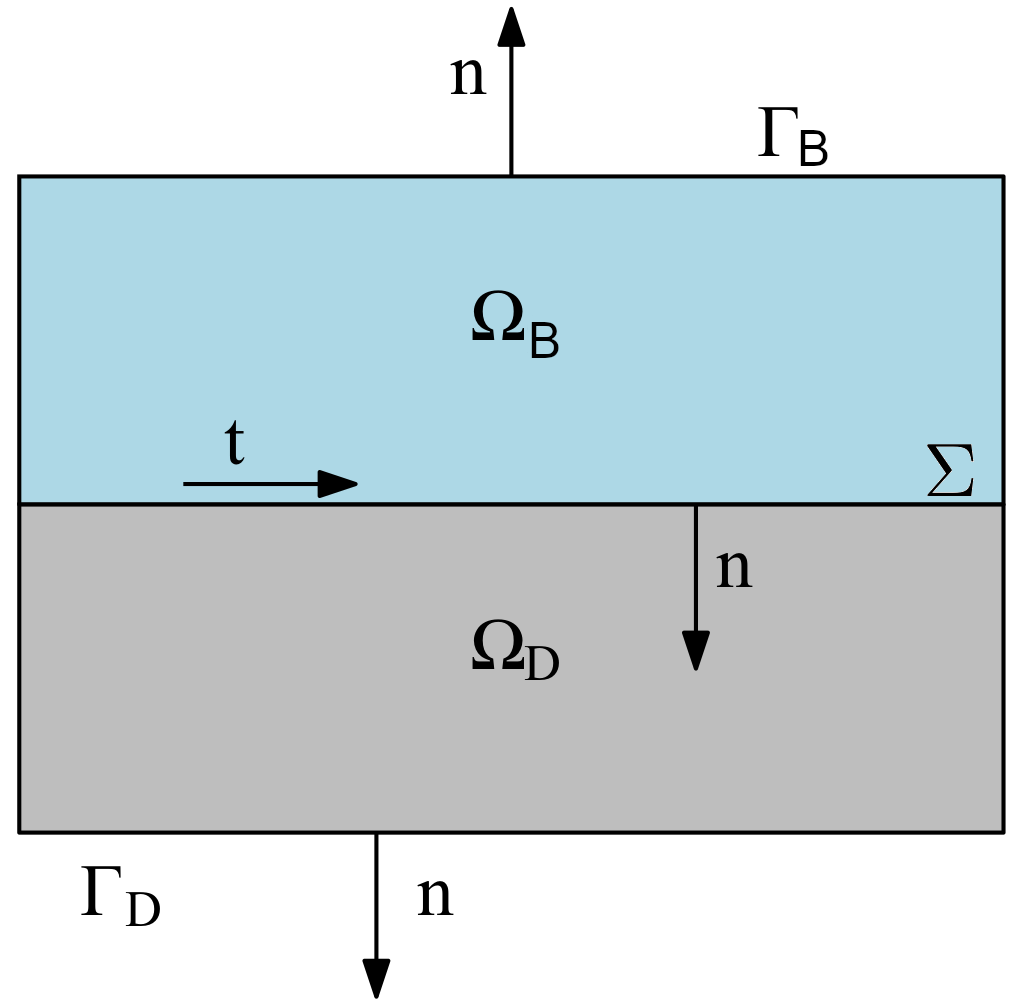}

\vspace{-0.2cm}
\caption{Sketch of a 2D geometry of the coupled Brinkman--Forchheimer/Darcy model}
\label{fig:dominio-2d}
\end{figure}

To couple the Brinkman--Forchheimer and the Darcy models, we proceed as in \cite{cd2023} (see similar approaches in \cite{Ehrhardt:2012:PICP,Dumitrache:2012:ATN,lcs2021}), and consider transmission conditions that impose, respectively, the mass conservation and continuity of momentum across the interface $\Sigma$:
\begin{equation}\label{eq:transmission-condition}
\bu_\rB\cdot\bn = \bu_\rD\cdot\bn 
\qan
\ds \bsi_\rB\bn = -p_\rD\bn \qon \Sigma \,.
\end{equation}

Notice that, according to the compressibility conditions, the boundary conditions on $\bu_\rD$ and $\bu_\rB$, and the principle of mass conservation (cf. first equation in \eqref{eq:transmission-condition}), $g_\rD$ must satisfy the compatibility condition:
\begin{equation*}
\int_{\Omega_\rD} g_\rD = 0.
\end{equation*}
In addition, we also observe that other boundary conditions can be considered. For example, similarly to \cite{do2017}, one could impose
\begin{equation}\label{eq:alternative-BC}
\begin{array}{c}
\ds \bu_\rB = \0 \qon \Gamma^d_\rB\,,\quad \bsi_\rB\bn = \0 \qon \Gamma^n_\rB\,, \\[2ex]
\ds p_\rD = 0 \qon \Gamma^d_\rD\,,\quad \bu_\rD\cdot\bn = 0 \qon \Gamma^n_\rD \,,
\end{array}
\end{equation}
where $\Gamma^d_\rB\cup \Gamma^n_\rB = \Gamma_\rB, \Gamma^d_\rD\cup \Gamma^n_\rD=\Gamma_\rD$, 
and $\Gamma^d_\rB\cap \Sigma = \emptyset, \Gamma^d_\rD\cap \Sigma = \emptyset$.
The analysis studied in this work can be extended with minor modifications to the case when \eqref{eq:alternative-BC} are used.
However, for the sake of simplicity, we focus on \eqref{eq:BF-model}--\eqref{eq:Darcy-model} for the analysis,
and consider \eqref{eq:alternative-BC} in one of the numerical examples in Section \ref{sec:numerical-results}.


\subsection{The variational formulation}

We first recall the following notations from \cite[Section 2.2]{cd2023}. Given $\star \in \{\rB, \rD\}$, let
\begin{equation*}
(p,q)_\star := \int_{\Omega_\star} p\,q,\quad 
(\bu,\bv)_\star := \int_{\Omega_\star} \bu\cdot\bv, \qan
(\bsi,\btau)_\star := \int_{\Omega_\star} \bsi:\btau,
\end{equation*}
where, given two arbitrary tensors $\bsi$ and $\btau$, $\bsi:\btau = \tr(\bsi^\rt\btau)=\ds\sum_{i,j=1}^2 \sigma_{ij}\tau_{ij}$.
Furthermore, we consider the Hilbert space
\begin{equation*}
\bH(\div;\Omega_\rD) := \Big\{\bv_\rD\in\bL^2(\Omega_\rD):\quad \div(\bv_\rD) \in\L^2(\Omega_\rD) \Big\},
\end{equation*}
endowed with the norm
$\|\bv_\rD\|^2_{\div;\Omega_\rD} := \|\bv_\rD\|^2_{0,\Omega_\rD} + \|\div(\bv_\rD)\|^2_{0,\Omega_\rD}$,
and the following subspaces of $\bH^1(\Omega_\rB)$ and $\bH(\div;\Omega_\rD)$, respectively,
\begin{align*}
\ds \bH^1_{\Gamma_\rB} (\Omega_\rB) & \,:=\, \Big\{\bv_\rB\in\bH^1(\Omega_\rB):\quad \bv_\rB = \0 \qon \Gamma_\rB \Big\}, \\[1ex]
\ds \bH_{\Gamma_\rD} (\div;\Omega_\rD) & \,:=\, \Big\{ \bv_\rD\in\bH(\div;\Omega_\rD):\quad \bv_\rD\cdot\bn = 0 \qon \Gamma_\rD \Big\} \,.
\end{align*}
In addition, we write $\Omega := \Omega_\rB \cup \Sigma \cup \Omega_\rD$,
and define $p:= p_\rB \chi_\rB + p_\rD \chi_\rD$, with $\chi_\star$ being the characteristic function:
\begin{equation*}
\chi_\star := \left\{\begin{array}{lll}
1 & \mbox{ in } & \Omega_\star, \\ [1ex]
0 & \mbox{ in } & \Omega\setminus \ov{\Omega}_\star,
\end{array}\right.
\quad\mbox{for }\, \star\in\{\rB,\rD\}\,,
\end{equation*}
and introduce the space $\ds\L^2_0(\Omega) := \Big\{q\in \L^2(\Omega):\quad \int_\Omega q =0 \Big\}$.

Next, for the sake of clarity, we group the spaces and unknowns as follows:
\begin{equation*}
\begin{array}{c}
\bH := \bH^1_{\Gamma_\rB} (\Omega_\rB)\times \bH_{\Gamma_\rD} (\div;\Omega_\rD), \quad 
\bQ := \L^2_0(\Omega)\times \H^{1/2}(\Sigma), \\ [1ex]
\ds \bu := (\bu_\rB,\bu_\rD)  \in \bH, \quad (p,\lambda) \in \bQ,
\end{array}
\end{equation*}
where $\lambda:= p_\rD|_\Sigma\in \H^{1/2}(\Sigma)$ is an additional unknown.
In turn, the corresponding norms associated to $\bH$ and $\bQ$ are given by
\begin{equation*}
\|\bv\|^2_{\bH} \,:=\, \|\bv_\rB\|^2_{1,\Omega_\rB} + \|\bv_\rD\|^2_{\div;\Omega_\rD} \quad \forall\,\bv:=(\bv_\rB,\bv_\rD)\in \bH 
\end{equation*}
and
\begin{equation*}
\|(q,\xi)\|^2_{\bQ} \,:=\, \|q\|^2_{0,\Omega} + \|\xi\|^2_{1/2;\Sigma} \quad \forall\,(q,\xi) \in \bQ\,.
\end{equation*}
Thus, we arrive at the mixed variational formulation: 
Find $(\bu,(p,\lambda)) \in \bH\times\bQ$, such that
\begin{equation}\label{eq:mixed-variational-formulation}
\begin{array}{llll}
[\ba(\bu),\bv] + [\bb(\bv),(p,\lambda)] & = & [\f,\bv] & \forall\,\bv\in \bH, \\ [2ex]
[\bb(\bu),(q,\xi)] & = & [\g,(q,\xi)] & \forall\,(q,\xi) \in \bQ,
\end{array}
\end{equation}
where, the operator $\ba : \bH \to \bH'$ is defined by
\begin{equation}\label{eq:definition-a}
[\ba(\bu),\bv] 
\,:=\, \mu (\nabla\bu_\rB,\nabla\bv_\rB)_\rB 
+ (\bK^{-1}_\rB\bu_\rB,\bv_\rB)_\rB
+ \tF\,(|\bu_\rB|^{\rho-2}\bu_\rB,\bv_\rB)_\rB  
+ \left(\bK^{-1}_\rD \bu_\rD,\bv_\rD\right)_\rD\,, 
\end{equation}
whereas the operator $\bb:\bH\to\bQ'$ is given by
\begin{equation}\label{eq:definition-b}
[\bb(\bv),(q,\xi)] := - (q,\div(\bv_\rB))_\rB - (q,\div(\bv_\rD))_\rD + \pil \bv_\rB\cdot\bn - \bv_\rD\cdot\bn,\xi \pir_\Sigma.
\end{equation}
In turn, the functionals $\f\in \bH'$ and $\g\in \bQ'$ are defined by
\begin{equation}\label{eq:definition-rhs}
[\f,\bv] := (\f_\rB,\bv_\rB)_\rB + (\f_\rD,\bv_\rD)_\rD \qan [\g,(q,\xi)] := -(g_\rD,q)_\rD.
\end{equation}
In all the terms above, $[\,\cdot,\cdot\,]$ denotes the duality pairing induced by the corresponding operators.
The solvability analysis for \eqref{eq:mixed-variational-formulation} is established in \cite[Theorem 3.5]{cd2023}.
In particular, we recall the following {\it a priori} estimate for later use:
\begin{equation}\label{eq:a-priori-bound-u}
\|(\bu_\rB,\bu_\rD)\|_\bH \,=\, \|\bu\|_{\bH} \,\leq\, C\,\Big( \|\f_\rB\|_{0,\Omega_\rB} + \|\f_\rD\|_{0,\Omega_\rD} 
+ \|g_\rD\|_{0,\Omega_\rD} + \|g_\rD\|^{\rho-1}_{0,\Omega_\rD} \Big)\,, 
\end{equation}
where $C$ is a positive constant, independent of the solution.


\subsection{The conforming finite element method}

Let $\cT^\rB_h$ and $\cT^\rD_h$ be respective triangulations of the domains $\Omega_\rB$ 
and $\Omega_\rD$ formed by shape-regular triangles, denote by $h_\rB$ 
and $h_\rD$ their corresponding mesh sizes, and let $h : = \max\big\{ h_\rB, h_\rD \big\}$. 
Assume that $\cT^\rB_h$ and $\cT^\rD_h$ match on $\Sigma$ so that $\cT_h := \cT^\rB_h \cup \cT^\rD_h$ is a triangulation of 
$\Omega := \Omega_\rB \cup \Sigma \cup \Omega_\rD$. 
Then, given an integer $\ell\geq 0$ and a subset $S$ of $\R^2$, we denote by $\rP_\ell(S)$ the space of polynomials of total degree at most $\ell$ defined on $S$.
For each $T\in \cT^\rD_h$ we consider the local Raviart--Thomas space of the lowest order \cite{Raviart-Thomas}:
\begin{equation*}
\bRT_0(T) := [\rP_0(T)]^2 \oplus \rP_0(T)\,\bx \,,
\end{equation*}
where $\bx:=(x_1,x_2)^\rt$ is a generic vector of $\R^2$.
In addition, for each $T\in \cT^\rB_h$ we denote by $\bBR(T)$ the local Bernardi--Raugel space \cite{br1985}:
\begin{equation*}
\bBR(T) := [\rP_1(T)]^2 \oplus \mathrm{span}\Big\{ \eta_2\eta_3\bn_1,\eta_1\eta_3\bn_2,\eta_1\eta_2\bn_3 \Big\},
\end{equation*}
where $\big\{ \eta_1, \eta_2, \eta_3\big\}$ are the baricentric coordinates of $T$, and $\big\{\bn_1, \bn_2, \bn_3\big\}$ are the unit outward normals to the opposite sides of the corresponding vertices of $T$. Hence, we define the following conforming finite element subspaces:
\begin{align*}
\bH_h (\Omega_\rB)  & \,:= \, \Big\{ \bv\in\bH^1(\Omega_\rB) : \quad \bv|_T \in \bBR(T),\quad \forall\, T\in\cT^\rB_h \Big\} \,, \\
\bH_h(\Omega_\rD)  & \,:=\, \Big\{ \bv\in \bH(\div;\Omega_\rD) : \quad \bv|_T \in \bRT_0(T),\quad \forall\, T\in\cT^\rD_h \Big\} \,, \\
\L_h(\Omega) & \,:=\, \Big\{ q\in \L^2(\Omega) :\quad q|_T\in \rP_0(T), \quad \forall\, T\in \cT_h \Big\} \,.
\end{align*}
Then, the finite element subspaces for the velocities and pressure are, respectively,
\begin{align}
\bH_{h,\Gamma_\rB}(\Omega_\rB) & \,:=\, \bH_h(\Omega_\rB)\cap \bH^1_{\Gamma_\rB}(\Omega_\rB) \,,  \nonumber \\[1ex] 
\bH_{h,\Gamma_\rD}(\Omega_\rD) & \,:=\, \bH_h(\Omega_\rD)\cap \bH_{\Gamma_\rD}(\div;\Omega_\rD) \,, \label{eq:FEM-1} \\[1ex]
\L_{h,0}(\Omega) & \,:=\, \L_h(\Omega)\cap \L^2_0(\Omega) \,. \nonumber
\end{align}

Next, to introduce the finite element subspace of $\H^{1/2}(\Sigma)$, 
we denote by $\Sigma_h$ the partition of $\Sigma$ inherited from $\cT^\rD_h$ (or $\cT^\rB_h$)
and assume without loss of generality, that the number of edges of $\Sigma_h$ is even.
Then, we let $\Sigma_{2h}$ be the partition of 
$\Sigma$ that arises by joining pairs of adjacent edges of $\Sigma_h$, 
denote the resulting edges still by $e$, and define $h_\Sigma := 
\max\{ h_e:\quad e\in \Sigma_{2h} \}$. If the number of edges of $\Sigma_h$ 
were odd, we first reduce it to the even case by joining any pair of two adjacent 
elements, construct $\Sigma_{2h}$ from this reduced partition, and define $h_\Sigma$
as indicated above.
Then, we define the following finite element subspace for $\lambda\in \H^{1/2}(\Sigma)$
\begin{equation}\label{eq:FEM-2}
\Lambda_h(\Sigma) := \Big\{ \xi_h\in \cC(\Sigma) :\quad \xi_h|_e\in \rP_1(e) \quad \forall\,e\in \Sigma_{2h} \Big\} \,.
\end{equation}

In this way, grouping the unknowns and spaces as follows:
\begin{equation*}
\begin{array}{c}
\bH_h := \bH_{h,\Gamma_\rB}(\Omega_\rB)\times \bH_{h,\Gamma_\rD} (\Omega_\rD), \quad \bQ_h := \L_{h,0}(\Omega)\times \Lambda_h(\Sigma), \\ [1ex]
\ds \bu_h := (\bu_{\rB,h}, \bu_{\rD,h})  \in \bH_h, \quad (p_h, \lambda_h) \in \bQ_h,
\end{array}
\end{equation*}
where $p_h:= p_{\rB,h} \chi_\rB + p_{\rD,h} \chi_\rD$, the Galerkin scheme for \eqref{eq:mixed-variational-formulation} reads:  
Find $(\bu_h,(p_h,\lambda_h)) \in \bH_h\times\bQ_h$, such that
\begin{equation}\label{eq:discrete-mixed-formulation}
\begin{array}{llll}
[\ba(\bu_h),\bv_h] + [\bb(\bv_h),(p_h,\lambda_h)] & = & [\f,\bv_h] & \forall\,\bv_h:=(\bv_{\rB,h},\bv_{\rD,h}) \in \bH_h \,, \\ [2ex]
[\bb(\bu_h),(q_h,\xi_h)] & = & [\g,(q_h,\xi_h)] & \forall\,(q_h,\xi_h) \in \bQ_h\,.
\end{array}
\end{equation}
The solvability analysis and {\it a priori} error bounds for \eqref{eq:discrete-mixed-formulation} are established in \cite[Theorems 4.3 and 4.6]{cd2023}.
In particular, we recall the following {\it a priori} estimate for later use: 
\begin{equation}\label{eq:a-priori-bound-uh}
\|(\bu_{\rB,h}, \bu_{\rD,h})\|_\bH \,=\, \|\bu_h\|_{\bH} \,\leq\, \wh{C}\,\Big( \|\f_\rB\|_{0,\Omega_\rB} + \|\f_\rD\|_{0,\Omega_\rD} 
+ \|g_\rD\|_{0,\Omega_\rD} + \|g_\rD\|^{\rho-1}_{0,\Omega_\rD} \Big)\,, 
\end{equation}
where $\wh{C}$ is a positive constant, independent of $h$ and the solution.


\section{A residual-based \textit{a posteriori} error estimator}\label{sec:a-posteriori-error-analysis}

In this section we derive a reliable and efficient residual-based {\it a posteriori} error estimator 
for the two-dimensional version of the Galerkin scheme \eqref{eq:discrete-mixed-formulation}.
To this end, from now on we employ the notations and results from Appendix \ref{app:preliminaries-for-reliability}.
Recalling that $(\bu_h,(p_h,\lambda_h)) \in\bH_h\times \bQ_h$ 
is the unique solution of the discrete problem \eqref{eq:discrete-mixed-formulation}, 
and denoting $p_{\rB,h} := p_h|_{\Omega_\rB}$ and $p_{\rD,h} := p_h|_{\Omega_\rD}$,
we define for each $T\in\cT^\rB_h$ the local error indicator 
\begin{equation}\label{eq:local-estimator-B}
\begin{array}{l}
\ds \Theta^2_{\rB,T} \,:=\, \|\div(\bu_{\rB,h})\|^2_{0,T} 
+ h^2_T \left\|\f_\rB + \bdiv(\bsi_{\rB,h}) - \bK^{-1}_\rB\bu_{\rB,h} - \tF\,|\bu_{\rB,h}|^{\rho-2}\bu_{\rB,h} \right\|^2_{0,T} \\[3ex]
\ds\quad + \sum_{e\in\cE(T)\cap\cE_h(\Omega_\rB)} h_e \big\|\jump{\bsi_{\rB,h}\bn}\big\|^2_{0,e} 
+ \sum_{e\in\cE(T)\cap\cE_h(\Sigma)} h_e \left\|\bsi_{\rB,h}\bn + \lambda_h\bn \right\|^2_{0,e},
\end{array}
\end{equation}
where, $\jump{ \cdot }$ denotes the jump operator defined in \eqref{eq:jump-tensor-n}, and
\begin{equation}\label{eq:stress-Bh-definition}
\bsi_{\rB,h} \,:=\, -\,p_{\rB,h}\bbI + \mu\nabla\bu_{\rB,h} \quad \mbox{on each}\quad T\in\cT^\rB_h \,.
\end{equation}
Similarly, for each $T\in\cT^\rD_h$ we set
\begin{equation}\label{eq:local-estimator-D}
\begin{array}{l}
\ds \Theta^2_{\rD,T} \,:=\, \|g_\rD - \div(\bu_{\rD,h})\|^2_{0,T} 
+ h^2_T \left\|\f_\rD - \bK^{-1}_\rD \bu_{\rD,h}\right\|^2_{0,T} \\ [4ex]
\ds\quad +\,\, h^2_T\big\|\rot\left(\f_\rD - \bK^{-1}_\rD\bu_{\rD,h}\right)\big\|^2_{0,T}
+ \sum_{e\in \cE(T)\cap \cE_h(\Omega_\rD)} h_e\,\left\|\jump{\left(\f_\rD - \bK^{-1}_\rD\bu_{\rD,h}\right)\cdot\bt}\right\|^2_{0,e} \\[4ex]
\ds\quad +\, \sum_{e\in \cE(T)\cap \cE_h(\Sigma)} \Bigg\{ h_e\,\left\|\left(\f_\rD - \bK^{-1}_\rD\bu_{\rD,h}\right)\cdot\bt - \frac{d\,\lambda_h}{d\,\bt} \right\|^2_{0,e} \\[3ex]
\ds\hspace{3.2cm} +\,\, h_e\,\|\lambda_h - p_{\rD,h}\|^2_{0,e}
+ h_e\,\|\bu_{\rB,h}\cdot\bn - \bu_{\rD,h}\cdot\bn\|^2_{0,e} \Bigg\} \,,
\end{array}
\end{equation}
so that the global {\it a posteriori} error estimator is given by
\begin{equation}\label{eq:global-estimator}
\Theta_{\tBFD} \,:=\, \left\{\sum_{T\in\cT^\rB_h} \Theta^2_{\rB,T} 
+ \sum_{T\in\cT^\rD_h} \Theta^2_{\rD,T} \right\}^{1/2} \,.
\end{equation}
Notice that the third term defining $\Theta^2_{\rD,T}$,
require that $\f_\rD\in \bH^{1}(T)$ for each $T\in\cT^\rD_h$. 
Note also that the second term in \eqref{eq:local-estimator-D} can be replaced by $ h^2_T \left\|\nabla p_{\rD,h} - \big(\f_\rD - \bK^{-1}_\rD \bu_{\rD,h}\big) \right\|^2_{0,T}$, since $p_{\rD,h}|_T\in \rP_0(T)$, and thus $\nabla p_{\rD,h} = \0$ in $T$ for each $T\in\cT^\rD_h$.
This remark is made to emphasize the residual nature of the estimator.

The main goal of the present section is to establish, under suitable assumptions, the existence of positive constants $C_{\tt eff}$ and $C_{\tt rel}$, independent of the meshsizes and the continuous and discrete solutions, such that
\begin{equation}\label{eq:efficience-reliable}
C_{\tt eff}\,\Theta_{\tBFD} \,+\, {\tt h. o. t} 
\,\leq\, \|\bu - \bu_h\|_{\bH} \,+\, \|(p,\lambda) - (p_h,\lambda_h)\|_{\bQ}
\,\leq\, C_{\tt rel}\,\Theta_{\tBFD}\,,
\end{equation}
where ${\tt h. o. t.}$ is a generic expression denoting one or several terms of higher order. 
The upper and lower bounds in \eqref{eq:efficience-reliable}, which are known as the reliability and efficiency of $\Theta_{\tBFD}$, are derived below in Sections \ref{sec:reliability} and \ref{sec:efficiency}, respectively.

\subsection{Reliability of the \textit{a posteriori} error estimator}\label{sec:reliability}

The main result of this section, which establishes the reliability of $\Theta_\tBFD$, reads as follows.
\begin{thm}\label{thm:reliability}
There exists a constant $C_{\tt rel}>0$, independent of $h$, such that
\begin{equation}\label{eq:reliability-estimate}
\|\bu - \bu_h\|_{\bH} \,+\, \|(p,\lambda) - (p_h,\lambda_h)\|_{\bQ}
\,\leq\, C_{\tt rel}\,\Theta_\tBFD \,.
\end{equation}
\end{thm}

We begin the proof of Theorem \ref{thm:reliability} with a preliminary estimate 
for the global error in \eqref{eq:reliability-estimate}. 
Indeed, proceeding analogously to \cite[Section 3.1]{cgo2023} 
(see also \cite[Section 5.1]{cgos2021}), we first introduce the residual functionals 
$\cR_{\f} : \bH\to \R$ and $\cR_{\g} : \bQ\to \R$, defined by
\begin{equation}\label{eq:residue-Rf}
\cR_{\f}(\bv) \,:=\, [\f,\bv] - [\ba(\bu_h),\bv] - [\bb(\bv),(p_h,\lambda_h)] 
\quad \forall\, \bv\in \bH\,,
\end{equation}
and
\begin{equation}\label{eq:residue-Rg}
\cR_{\g}(q,\xi) \,:=\, [\g,(q,\xi)] - [\bb(\bu_{h}),(q,\xi)] 
\quad \forall\,(q,\xi)\in \bQ\,,
\end{equation}
respectively, which, according to the first and second equations of the discrete 
problem \eqref{eq:discrete-mixed-formulation}, satisfy
\begin{equation}\label{eq:R-hat-vanish-in-Hh-Qh}
\cR_{\f}(\bv_h) \,=\, 0 \quad\forall\,\bv_h \in \bH_h
\qan
\cR_{\g}(q_h,\xi_h) \,=\, 0 \quad \forall\, (q_h,\xi_h) \in \bQ_h\,. 
\end{equation}
In addition, denoting $\bV$ as the kernel of the bilinear form $\bb$ (cf. \eqref{eq:definition-b}), we observe that as a corollary of \cite[Lemma 3.3]{cd2023}, taking in particular $\bu=\bx - \by, \bv=\0\in \bV$ and $\bz=\by\in \bH$ in \cite[eq. (3.16)]{cd2023}, the operator $\ba$ (cf. \eqref{eq:definition-a}) satisfies the following strong monotonicity property
\begin{equation}\label{eq:strong-monotonicity-on-V}
[\ba(\bx) - \ba(\by),\bx - \by]
\,\ge\, \gamma_{\tBFD} \, \|\bx - \by\|^2_{\bH} \,,
\end{equation}
for all $\bx, \, \by \in \bH$ such that $\bx - \by\in \bV$, with the same constant $\gamma_{\tBFD}:=\min\{ \mu, C_{\bK_\rB}, C_{\bK_\rD} \}$ from \cite[eq. (3.16)]{cd2023} (cf. \eqref{eq:permeability-constrain}). 

The announced preliminary result is established as follows.
\begin{lem}\label{lem:preliminary-a-priori-bound}
Given $\rho\in [3,4]$.	
Then, there exists a constant $C>0$, independent of $h$, such that
\begin{equation}\label{eq:a-priori-bound}
\|\bu - \bu_h\|_{\bH} \,+\, \|(p,\lambda) - (p_h,\lambda_h)\|_{\bQ}
\,\leq\, C\,\Big\{ \|\cR_{\f}\|_{\bH'} + \|\cR_{\g}\|_{\bQ'} + \|\cR_{\g}\|^{\rho-1}_{\bQ'} \Big\} \,.
\end{equation}
\end{lem}
\begin{proof}
We proceed as in \cite[Lemma 3.2]{cgo2023}.	
Indeed, we observe from \eqref{eq:mixed-variational-formulation} and the definitions of $\cR_{\f}$ 
and $\cR_{\g}$ (cf. \eqref{eq:residue-Rf} and \eqref{eq:residue-Rg}) that
\begin{equation}\label{eq:estimate-A}
[\ba(\bu) - \ba(\bu_h),\bv] + [\bb(\bv),(p,\lambda) - (p_h,\lambda_h)] =  \cR_{\f}(\bv)
\qquad \forall\, \bv\in \bH\,,
\end{equation}
and
\begin{equation}\label{eq:estimate-B}
[\bb(\bu - \bu_h),(q,\xi)] = \cR_{\g}(q,\xi) \qquad \forall \, (q,\xi)\in \bQ \,.
\end{equation}
Thus, we employ the continuous inf-sup condition for $\bb$, which holds with a constant $\beta$
(cf. \cite[Lemma 3.4]{cd2023}), the converse implication of the equivalence 
provided in \cite[Lemma A.42]{Ern-Guermond}, and \eqref{eq:estimate-B}, to deduce that 
there exists $\bw := (\bw_\rB,\bw_\rD) \in \bH$ such that
\begin{equation}\label{eq:b-zuh-properties}
\bb(\bw) \,=\, \bb(\bu - \bu_h) \,=\, \cR_{\g} \qan 
\|\bw\|_{\bH} \,\leq\, \frac{1}{\beta}\,\|\cR_{\g}\|_{\bQ'} \,.
\end{equation}
It follows that the error $\bu - \bu_h$ can be decomposed as
\begin{equation}\label{eq:error-u-decomposition}
\bu - \bu_h = \bz + \bw\,,
\end{equation}
with $\bz := \bu - \bu_h - \bw \in \bV$.
Then, taking $\bv = \bz$ in \eqref{eq:estimate-A}, we find that
\[
[\ba(\bu) - \ba(\bu_h),\bz] =  \cR_{\f}(\bz)\,,
\]
and hence, subtracting and adding $\ba(\bu)$, we obtain
\begin{equation}\label{eq:auxiliary-equation-1}
[\ba(\bu - \bw) - \ba(\bu_h), \bz] \,=\,
[\ba(\bu - \bw) - \ba(\bu), \bz] \,+\, \cR_{\f}(\bz) \,.
\end{equation} 
Then, applying \eqref{eq:strong-monotonicity-on-V} to $\bx = \bu - \bw$ and $\by = \bu_h$, 
and using \eqref{eq:auxiliary-equation-1}, we find that
\[
\gamma_{\tBFD} \, \|\bz\|^2_{\bH}\,\le\, 
\|\ba(\bu - \bw) - \ba(\bu)\|_{\bH'}\,\|\bz\|_{\bH}
+ \|\cR_{\f}\|_{\bH'}\,\|\bz\|_{\bH}\,,
\]
from which, making use of the continuity of $\ba$ (cf. eq. (3.11) in \cite[Lemma 3.2]{cd2023}), which involves a constant $L_\tBFD$ depending on $\tF, \bK_\rD, \bK_\rB$, and $\mu$, 
and then performing simple algebraic computations, we obtain
\begin{equation*}
\gamma_{\tBFD}\,\|\bz\|_\bH 
\,\le\, L_{\tBFD}\,\Big\{ \big(1 + 2\,\|\bu_\rB\|^{\rho-2}_{1,\Omega_\rB} \big) \|\bw_\rB\|_{1,\Omega_\rB} 
+ \|\bw_\rD\|_{\div;\Omega_\rD}
+ \|\bw_\rB\|^{\rho-1}_{1,\Omega_\rB} \Big\} 
+ \|\cR_{\f}\|_{\bH'} \,.
\end{equation*}
The above estimate, together with the fact that $\|\bu_\rB\|_{1,\Omega_\rB}$ is bounded by data (cf. \eqref{eq:a-priori-bound-u}), yields
\begin{equation}\label{eq:auxiliary-bound-2}
\|\bz\|_{\bH} \,\le\, 
c_1\,\Big\{ \|\cR_{\f}\|_{\bH'} + \|\bw\|_{\bH} 
+ \|\bw\|^{\rho-1}_{\bH}  \Big\} \,,
\end{equation}
with $c_1>0$ independent of $h$, and hence, using \eqref{eq:b-zuh-properties}, 
\eqref{eq:error-u-decomposition} and \eqref{eq:auxiliary-bound-2}, we conclude that
\begin{equation}\label{eq:error-bound-u-uh}
\|\bu - \bu_h\|_{\bH} \,\le\, \|\bz\|_{\bH} + \|\bw\|_{\bH}
\,\le\, c_2\,\Big\{ \|\cR_{\f}\|_{\bH'} + \|\cR_{\g}\|_{\bQ'} + \|\cR_{\g}\|^{\rho-1}_{\bQ'} \Big\} \,,
\end{equation}
with $c_2>0$ depending only on $L_{\tBFD}$, $\gamma_{\tBFD}, \beta$, and data.
On the other hand, applying the continuous inf-sup condition 
for $\bb$ (cf. \cite[Lemma 3.4]{cd2023})
to $(p,\lambda) - (p_h,\lambda_h)$, employing the identity \eqref{eq:estimate-A} 
to express $[\bb(\bv),(p,\lambda) - (p_h,\lambda_h)]$, 
and using again the continuity of $\ba$ (cf. \cite[Lemma 3.2]{cd2023}), we deduce that
\begin{equation*}
\begin{array}{l}
\ds
\beta\,\|(p,\lambda) - (p_h,\lambda_h)\|_{\bQ} \,\le\, 
\sup_{\substack{\bv\in\bH\\ \bv\neq \0}} 
\frac{-[\ba(\bu) - \ba(\bu_h),\bv] + \cR_{\f}(\bv)}{\|\bv\|_\bH} \\[2ex]
\ds
\le\,  L_{\tBFD}\,\Big\{ 1+ \big(\|\bu_\rB\|_{1,\Omega_\rB} + \|\bu_{\rB,h}\|_{1,\Omega_\rB}\big)^{\rho-2} \Big\}\|\bu - \bu_{h}\|_{\bH}
+ \|\cR_{\f}\|_{\bH'} \,,
\end{array}
\end{equation*}
which, along with the fact that both $\|\bu_\rB\|_{1,\Omega_\rB}$ and $\|\bu_{\rB,h}\|_{1,\Omega_\rB}$ 
are bounded by data (cf. \eqref{eq:a-priori-bound-u} and \eqref{eq:a-priori-bound-uh}), and some algebraic manipulations, imply
\begin{equation}\label{eq:error-bound-p-lambda-ph-lambdah}
\|(p,\lambda) - (p_h,\lambda_h)\|_{\bQ}
\,\leq\, c_3\,\Big\{ \|\bu - \bu_{h}\|_{\bH} + \|\cR_{\f}\|_{\bH'} \Big\} \,,
\end{equation}
with $c_3>0$ depending only on $L_{\tBFD}, \beta$, and data.
Therefore, the estimate \eqref{eq:a-priori-bound} follows from \eqref{eq:error-bound-u-uh} 
and \eqref{eq:error-bound-p-lambda-ph-lambdah}, thus ending the proof.
\end{proof}	

\medskip

We remark here that, similarly to \cite[eq. (3.29)]{cgo2023}, when $\|\cR_{\g}\|_{\bQ'} < 1$
the term $\|\cR_{\g}\|^{\rho-1}_{\bQ'}$, with $\rho\in [3,4]$, is dominated by $\|\cR_{\g}\|_{\bQ'}$,
whence the former can be neglected, it follows from \eqref{eq:a-priori-bound} that
\begin{equation}\label{eq:a-priori-bound-2}
\|\bu - \bu_h\|_{\bH} \,+\, \|(p,\lambda) - (p_h,\lambda_h)\|_{\bQ}
\,\leq\, C\,\Big\{ \|\cR_{\f}\|_{\bH'} + \|\cR_{\g}\|_{\bQ'} \Big\} \,,
\end{equation} 
with $C>0$, independent of $h$.
Note that when $\|\cR_{\g}\|_{\bQ'}>1$, the term $\|\cR_{\g}\|_{\bQ'}^{\rho-1}$, being dominant, 
will appears in \eqref{eq:a-priori-bound-2} instead of $\|\cR_{\g}\|_{\bQ'}$.  As a consequence, the reliability estimate in Lemmas \ref{lem:reliability-R3} and \ref{lem:reliability-R4}, and the 
local estimators $\Theta_{\rB,T}$ and $\Theta_{\rD,T}$ (cf. \eqref{eq:local-estimator-B}, \eqref{eq:local-estimator-D}) must be 
modified accordingly. The case $\|\cR_{\g}\|_{\bQ'}<1$ is assumed here for sake of simplicity. Nevertheless, being $\cR_{\g}$ a residual expression, it is expected to converge to $0$, which is somehow confirmed later on by the efficiency estimate and numerical results, so that the foregoing assumption seems 
quite reasonable.
Throughout the rest of this section, we provide suitable upper bounds for each one of the terms on the right-hand side of \eqref{eq:a-priori-bound-2}.
To this end, and based on the definitions of the operators $\ba$ and $\bb$ (cf. \eqref{eq:definition-a}, \eqref{eq:definition-b}) as well as the functionals $\f$ and $\g$ (cf. \eqref{eq:definition-rhs}), we first note that the functionals $\cR_{\f}$ and $\cR_{\g}$ can be decomposed as follows:
\begin{equation*}
\cR_{\f}(\bv) \,=\, \cR_1(\bv_\rB) + \cR_2(\bv_\rD) \qan
\cR_{\g}(q,\xi) \,=\, \cR_3(q) + \cR_4(\xi),
\end{equation*}
for all $\bv := (\bv_\rB,\bv_\rD)\in \bH$ and for all $(q,\xi)\in \bQ$, where
\begin{align}
\ds \cR_1(\bv_\rB) \,:=\, & \ds \left(\f_\rB - \bK^{-1}_\rB \bu_{\rB,h} - \tF\,|\bu_{\rB,h}|^{\rho-2}\bu_{\rB,h},\bv_\rB\right)_\rB \nonumber \\[1ex]
& \ds  -\,\, \mu\,(\nabla\bu_{\rB,h},\nabla\bv_\rB)_\rB 
+ (\div(\bv_\rB),p_h)_\rB - \pil\bv_\rB\cdot\bn,\lambda_h\pir_\Sigma \,, \label{def-R1} \\[1.5ex]
\ds \cR_2(\bv_\rD) \,:=\, & \ds \left(\f_\rD - \bK^{-1}_\rD\bu_{\rD,h}, \bv_\rD\right)_\rD + 
(\div(\bv_\rD),p_h)_\rD + \pil\bv_\rD\cdot\bn,\lambda_h\pir_\Sigma \,, \label{def-R2} \\[1.5ex]
\ds \cR_3(q) \,:=\, & \ds (\div(\bu_{\rB,h}),q)_\rB - (g_\rD - \div(\bu_{\rD,h}),q)_\rD \,, \label{def-R3} \\[1.5ex]
\ds \cR_4(\xi) \,:=\, & \ds -\pil\bu_{\rB,h}\cdot\bn - \bu_{\rD,h}\cdot\bn,\xi\pir_\Sigma \,. \label{def-R4}
\end{align}
In this way, it follows that
\begin{equation}\label{eq:upper-bound-residuals}
\|\cR_{\f}\|_{\bH'} + \|\cR_{\g}\|_{\bQ'} 
\,\leq\, \Big\{ \|\cR_1\|_{\bH^1_{\Gamma_\rB}(\Omega_\rB)'} 
+ \|\cR_2\|_{\bH_{\Gamma_\rD}(\div;\Omega_\rD)'} 
+ \|\cR_3\|_{\L^2_0(\Omega)'} + \|\cR_4\|_{\H^{1/2}(\Sigma)'} \Big\} \,.
\end{equation}
Thus, our next goal is to derive appropriate upper bounds for each term on the right-hand side of \eqref{eq:upper-bound-residuals}. We begin by obtaining the estimate for $\cR_3$ (cf. \eqref{def-R3}), which follows directly from the Cauchy--Schwarz inequality.
\begin{lem}\label{lem:reliability-R3}
There holds
\begin{equation*}
\|\cR_3\|_{\L^2_0(\Omega)'} 
\,\leq\, \left\{\sum_{T\in\cT^\rB_h} \|\div(\bu_{\rB,h})\|^2_{0,T} 
+ \sum_{T\in\cT^\rD_h} \|g_\rD - \div(\bu_{\rD,h})\|^2_{0,T} \right\}^{1/2}.
\end{equation*}	
\end{lem}

We now adapt a result taken from \cite{bg2010} in order to obtain an upper bound for $\cR_1$ (cf. \eqref{def-R1}).
\begin{lem}
There exists $C>0$, independent of the meshsizes, such that
\begin{equation*}
\|\cR_1\|_{\bH^1_{\Gamma_\rB}(\Omega_\rB)'} 
\,\leq\, C\,\left\{ \sum_{T\in\cT^\rB_h} \wh{\Theta}^2_{\rB,T} \right\}^{1/2},
\end{equation*}
where
\begin{align*}
\wh{\Theta}^2_{\rB,T} \,:=\, & h^2_T\,\left\|\f_\rB + \bdiv(\bsi_{\rB,h}) - \bK^{-1}_\rB\bu_{\rB,h} - \tF\,|\bu_{\rB,h}|^{\rho-2}\bu_{\rB,h}\right\|^2_{0,T} \\[2ex]
& +\,\,\sum_{e\in\cE(T)\cap\cE_h(\Omega_\rB)} h_e\,\big\|\jump{\bsi_{\rB,h}\bn}\big\|^2_{0,e} 
+ \sum_{e\in\cE(T)\cap\cE_h(\Sigma)} h_e\,\left\|\bsi_{\rB,h}\bn + \lambda_h\bn \right\|^2_{0,e}\,,
\end{align*}
and $\bsi_{\rB,h}$ is given by \eqref{eq:stress-Bh-definition}.
\end{lem}
\begin{proof}
We proceed similarly as in the proof of \cite[Lemma 3.4]{bg2010}, by replacing $\f_1, \bsi_{1,h}, \bu_{1,h}, \Omega_1$, 
and $\Gamma_2$ in there by $\f_\rB - \bK^{-1}_\rB\bu_{\rB,h} - \tF\,|\bu_{\rB,h}|^{\rho-2}\bu_{\rB,h}$, 
$\bsi_{\rB,h}$, $\bu_{\rB,h}$, $\Omega_\rB$, and $\Sigma$, respectively, 
and employing the local approximation properties of the Cl\'ement interpolation operator $\bcI^\rB_h:\bH^1(\Omega_\rB)\to \bX_h(\Omega_\rB)$ provided by Lemma \ref{lem:clement}.
We omit further details.
\end{proof}

Next, we derive the upper bound for $\cR_4$ (cf. \eqref{def-R4}), the functional acting on the interface $\Sigma$.
\begin{lem}\label{lem:reliability-R4}
There exists $C>0$, independent of the meshsizes, such that
\begin{equation*}
\|\cR_4\|_{\H^{1/2}(\Sigma)'} 
\,\leq\, C\,\left\{\sum_{e\in\cE_h(\Sigma)} h_e\,\|\bu_{\rB,h}\cdot\bn - \bu_{\rD,h}\cdot\bn\|^2_{0,e} \right\}^{1/2}.
\end{equation*}
\end{lem}
\begin{proof}
It follows similarly to the proof of \cite[eq. (3.23)]{bg2010} (see also \cite[eq. (3.8) in Lemma 3.2]{gos2011}), by replacing $\bu_{1,h}, \bu_{2,h}$, and $\Gamma_2$ therein with $\bu_{\rB,h}, \bu_{\rD,h}$, and $\Sigma$, respectively, employing \cite[Theorem 2]{carstensen1997}, and recalling that both $\Sigma_h$ and $\Sigma_{2h}$ are of bounded variation. Further details are omitted.
%
\end{proof}

Finally, we focus on deriving the upper bound for $\cR_2$, for which, given $\bv_\rD\in\bH_{\Gamma_\rD}(\div;\Omega_\rD)$, we consider its Helmholtz decomposition provided by Lemma~\ref{lem:Helmholtz-decomposition}. 
More precisely, we let $\bw_\rD\in\bH^1(\Omega_\rD)$ and $\beta_\rD\in\H^{1}_{\Gamma_\rD}(\Omega_\rD)$ (cf. \eqref{eq:H1-GammaD-space}) be such that
\begin{equation}\label{eq:Helmholtz-stability-estimate}
\bv_\rD = \bw_\rD + \bcurl(\beta_\rD) \qin \Omega_\rD \qan
\|\bw_\rD\|_{1,\Omega_\rD} + \|\beta_\rD\|_{1,\Omega_\rD} 
\,\leq\, \Chel\,\|\bv_\rD\|_{\div;\Omega_\rD}.
\end{equation}
In turn, similarly to \cite[Section 3.2]{agr2016}, we consider the finite element subspace of $\H^{1}_{\Gamma_\rD}(\Omega_\rD)$ given by
\begin{equation}\label{X-h-Gamma-D}
\X_{h,\Gamma_\rD} := \Big\{v\in\cC(\overline{\Omega}_\rD):\quad v|_T\in \rP_1(T) \quad \forall \, T\in\cT^\rD_h,\quad v = 0 \qon \Gamma_\rD \Big\},
\end{equation}
and introduce the Cl\'ement interpolator $\cI^\rD_h:\H^{1}_{\Gamma_\rD}(\Omega_\rD)\to \X_{h,\Gamma_\rD}$ (cf. \eqref{eq:Xh-Clement-space}).
In addition, recalling the Raviart--Thomas interpolator $\Pi_h:\bH^1(\Omega_\rD)\to \bH_{h,\Gamma_{\rD}}(\Omega_\rD)$ introduced in Appendix \ref{app:preliminaries-for-reliability}, we define 
\begin{equation}\label{eq:discrete-Helmholtz-decomposition}
\bv_{\rD,h} := \Pi_h(\bw_\rD) + \bcurl(\cI^\rD_h(\beta_\rD)) \in\bH_{h,\Gamma_\rD}(\Omega_\rD) \,,
\end{equation}
which can be interpreted as a discrete Helmholtz decomposition of $\bv_{\rD,h}$. Then, by testing \eqref{eq:R-hat-vanish-in-Hh-Qh} with $\bv_h=(\0,\bv_{\rD,h})$, we obtain $\cR_2(\bv_{\rD,h}) = 0$.
Denoting
\begin{equation*}
\widehat{\bw}_\rD := \bw_\rD - \Pi_h(\bw_\rD) \qan 
\widehat{\beta}_\rD := \beta_\rD - \cI^\rD_h(\beta_\rD) \,,
\end{equation*}
it follows from \eqref{eq:Helmholtz-stability-estimate} and \eqref{eq:discrete-Helmholtz-decomposition}, that
\begin{equation}\label{eq:R2-decomposition}
\cR_2(\bv_\rD) = \cR_2(\bv_\rD - \bv_{\rD,h}) 
\,=\, \cR_2(\widehat{\bw}_\rD) + \cR_2(\bcurl(\widehat{\beta}_\rD)) \,,
\end{equation}
where, according to the definition of $\cR_2$ (cf. \eqref{def-R2}), we have
\begin{equation}\label{eq:R2-hat-wD}
\cR_2(\wh{\bw}_\rD) 
\,=\, \left(\f_\rD - \bK^{-1}_\rD\bu_{\rD,h}, \wh{\bw}_\rD\right)_\rD 
+ (\div(\wh{\bw}_\rD),p_h)_\rD + \pil\wh{\bw}_\rD\cdot\bn,\lambda_h\pir_\Sigma 
\end{equation}
and
\begin{equation}\label{eq:R2-curl-hat-betaD}
\cR_2(\bcurl(\wh{\beta}_\rD)) \,=\, 
\left(\f_\rD - \bK^{-1}_\rD\bu_{\rD,h}, \bcurl(\wh{\beta}_\rD)\right)_\rD 
+ \pil\bcurl(\wh{\beta}_\rD)\cdot\bn,\lambda_h\pir_\Sigma \,.
\end{equation}
The following lemma establishes the residual upper bound for $\|\cR_2\|_{\bH_{\Gamma_\rD}(\div;\Omega_\rD)'}$.
\begin{lem}\label{lem:reliability-R2}
Assume that there exists a convex domain $\Xi$ such that $\Omega_\rD\subseteq \Xi$ and $\Gamma_\rD\subseteq\partial \Xi$. 
Assume further that $\f_\rD\in\bH^{1}(T)$ for each $T\in \cT^\rD_h$.
Then there exists $C>0$, independent of the meshsizes, such that
\begin{equation}\label{eq:R2-bound-ThetaD}
\|\cR_2\|_{\bH_{\Gamma_{\rD}}(\div;\Omega_\rD)'} 
\,\leq\, C\,\left\{ \sum_{T\in\cT^\rD_h} \wh{\Theta}^2_{\rD,T} \right\}^{1/2} \,,
\end{equation}
where 
\begin{equation}\label{eq:local-hat-ThetaD}
\begin{array}{l}
\ds \wh{\Theta}^2_{\rD,T} \,:=\, 
h^2_T \left\|\f_\rD - \bK^{-1}_\rD\bu_{\rD,h} \right\|^2_{0,T}
+  h^2_T\left\|\rot\left(\f_\rD - \bK^{-1}_\rD\bu_{\rD,h}\right)\right\|^2_{0,T} \\ [4ex]
\ds\quad +\,\sum_{e\in \cE(T)\cap \cE_h(\Omega_\rD)} h_e\,\left\|\jump{\left(\f_\rD - \bK^{-1}_\rD\bu_{\rD,h}\right)\cdot\bt}\right\|^2_{0,e} \\[4ex]
\ds\quad +\, \sum_{e\in \cE(T)\cap \cE_h(\Sigma)} \left\{ h_e\,\left\| (\f_\rD - \bK^{-1}_\rD\bu_{\rD,h})\cdot\bt - \frac{d\,\lambda_h}{d\,\bt} \right\|^2_{0,e} 
+ h_e\,\|\lambda_h - p_h\|^2_{0,e} \right\} \,.
\end{array}
\end{equation}
\end{lem}
\begin{proof}
The proof proceeds similarly as in \cite[Lemma 3.9]{gos2011} but now considering $\f_\rD - \bK^{-1}_\rD\bu_{\rD,h}$ instead of $\bK^{-1}\bu_{\rD,h}$. 
We provide details in what follows just for sake of completeness.
In fact, according to \eqref{eq:R2-decomposition}, we begin by estimating $\cR_2(\wh{\bw}_\rD)$.
Let us first observe that, for each $e\in \cE_h(\Sigma)$, the identity \eqref{eq:RT-interpolation}
and the fact that $p_{\rD,h}|_e\in \rP_0(e)$, yield $\ds\int_e (\wh{\bw}_\rD\cdot\bn)\,p_{\rD,h} = 0$.
Hence, locally integrating by parts the second term in \eqref{eq:R2-hat-wD},
we readily obtain
\begin{equation*}
\cR_2(\wh{\bw}_\rD) \,=\,
\int_\Omega \big(\f_\rD - \bK^{-1}_\rD\bu_{\rD,h}\big)\cdot \wh{\bw}_\rD 
\,+\, \sum_{e\in \cE_h(\Sigma)} \int_e \left(\lambda_h - p_{\rD,h}\right)\,\wh{\bw}_\rD\cdot\bn \,.
\end{equation*}
Thus, applying the Cauchy--Schwarz inequality along with the approximation properties of 
$\Pi_h$ (cf. \eqref{eq:property-Pi-1}--\eqref{eq:property-Pi-2} in Lemma \ref{lem:Phih-properties}), we find that
\begin{equation}\label{eq:reliability-estimate-R1-hat-zeta}
\big| \cR_2(\wh{\bw}_\rD) \big| \,\leq\, \wh{C}_1\,\Bigg\{ \sum_{T\in \cT^\rD_h} h^2_T\,\left\|\f_\rD - \bK^{-1}_\rD\bu_{\rD,h}\right\|^2_{0,T} \,+\, \sum_{e\in \cE_h(\Sigma)} h_e\,\|\lambda_h - p_{\rD,h}\|^2_{0,e} \Bigg\}^{1/2} \|\bw_\rD\|_{1,\Omega_\rD} \,.
\end{equation}

Next, we estimate $\cR_2(\bcurl(\wh{\beta}_\rD))$ (cf. \eqref{eq:R2-curl-hat-betaD}).
In fact, regarding its second term, noting that $\bcurl(\wh{\beta}_\rD)\cdot\bn = \nabla \wh{\beta}_\rD\cdot\bt$ on $\Sigma$, integrating by parts on $\Sigma$ (cf. \cite[Lemma 3.5, eq. (3.34)]{dgm2015}), and observing that $\frac{d\,\lambda_h}{d\,\bt}\in \L^2(\Sigma)$, we obtain
\begin{equation}\label{eq:identity-ubcurl-Gamma}
\pil\bcurl(\wh{\beta}_\rD)\cdot\bn,\lambda_h\pir_\Sigma 
\,=\, -\,\pil \frac{d\,\lambda_h}{d\,\bt}, \wh{\beta}_\rD \pir_\Sigma \,.
\end{equation}
In turn, locally integrating by parts the first term of $\cR_2(\bcurl(\wh{\beta}_\rD))$ and using the fact that $\wh{\beta}_\rD = 0$ on $\Gamma_{\rD}$ (cf. \eqref{eq:H1-GammaD-space} and \eqref{X-h-Gamma-D}), we get
\begin{equation*}
\begin{array}{l}
\ds \left(\f_\rD - \bK^{-1}_\rD\bu_{\rD,h}, \bcurl(\wh{\beta}_\rD)\right)_\rD
= \sum_{T\in \cT^\rD_h} \int_T \rot\left( \f_\rD - \bK^{-1}_\rD\bu_{\rD,h} \right)\,\wh{\beta}_\rD \\[3ex]
\ds\quad -\, \sum_{e\in \cE_h(\Omega_\rD)} \int_e \jump{\big( \f_\rD - \bK^{-1}_\rD\bu_{\rD,h} \big)\cdot\bt}\,\wh{\beta}_\rD
- \sum_{e\in \cE_h(\Sigma)} \int_e \big( \f_\rD - \bK^{-1}_\rD\bu_{\rD,h} \big)\cdot\bt\,\wh{\beta}_\rD \,,
\end{array}
\end{equation*}
which together with \eqref{eq:identity-ubcurl-Gamma}, the Cauchy--Schwarz inequality, and the approximation properties of $\cI^\rD_h$ (cf. Lemma \ref{lem:clement}), implies
\begin{equation}\label{eq:reliability-estimate-R1-curl-hat-xi}
\begin{array}{l}
\ds \big| \cR_2(\bcurl(\wh{\beta}_\rD)) \big| 
\,\leq\, \wh{C}_2\,\Bigg\{ \sum_{T\in \cT^\rD_h} h^2_T\,\left\|\rot\left(\f_\rD - \bK^{-1}_\rD\bu_{\rD,h}\right)\right\|^2_{0,T} \\[3ex]
\ds\quad +\, \sum_{e\in \cE_h(\Omega_\rD)} h_e\,\left\|\jump{\big( \f_\rD - \bK^{-1}_\rD\bu_{\rD,h} \big)\cdot\bt}\right\|^2_{0,e} \\[3ex]
\ds\quad +\, \sum_{e\in \cE_h(\Sigma)} h_e\,\left\|\big( \f_\rD - \bK^{-1}_\rD\bu_{\rD,h} \big)\cdot\bt - \frac{d\,\lambda_h}{d\,\bt}\right\|^2_{0,e} \Bigg\}^{1/2} 
\|\beta_\rD\|_{1,\Omega_\rD} \,.
\end{array}
\end{equation}
Finally, it is easy to see that \eqref{eq:R2-decomposition}, \eqref{eq:reliability-estimate-R1-hat-zeta}, \eqref{eq:reliability-estimate-R1-curl-hat-xi}, the stability estimate \eqref{eq:Helmholtz-stability-estimate}, and the definition of the local estimator $\wh{\Theta}^2_{\rD,T}$ (cf. \eqref{eq:local-hat-ThetaD}) give \eqref{eq:R2-bound-ThetaD}, which ends the proof.
\end{proof}

We end this section by stressing that the reliability estimate 
\eqref{eq:reliability-estimate} (cf. Theorem \ref{thm:reliability}) is a 
straightforward consequence of the auxiliar Lemma \ref{lem:preliminary-a-priori-bound} 
and Lemmas \ref{lem:reliability-R3}--\ref{lem:reliability-R2}, and the definition 
of the global estimator $\Theta_{\tBFD}$ (cf. \eqref{eq:global-estimator}).


\subsection{Efficiency of the \textit{a posteriori} error estimator}\label{sec:efficiency}

We now aim to establish the efficiency estimate of $\Theta_{\tBFD}$ (cf. \eqref{eq:global-estimator}).
For this purpose, we will make extensive use of the notation and results from Appendix \ref{app:preliminaries-for-efficiency}, as well as the original system of equations given by 
\eqref{eq:BF-model}, \eqref{eq:Darcy-model}, and \eqref{eq:transmission-condition}.
These equations can be derived from the mixed continuous formulation \eqref{eq:mixed-variational-formulation} by choosing 
suitable test functions and integrating by parts backwardly the corresponding equations.
The following theorem is the main result of this section.
\begin{thm}\label{th:efficiency} 
Assume, for simplicity, that $\f_\rB$ and $\f_\rD$ are piecewise polynomials.
Then, there exists a positive constant $C_{\tt eff}$, independent of $h$, such that 
\begin{equation}\label{eq:global-efficiency}
C_{\tt eff}\,\Theta_{\tBFD} \,+\, {\tt h. o. t.} 
\,\leq\, \|\bu - \bu_h\|_{\bH} \,+\, \|(p,\lambda) - (p_h,\lambda_h)\|_{\bQ} \,,
\end{equation}
where ${\tt h. o. t.}$ stands for one or several terms of higher order.
\end{thm}

Throughout this section we assume, without loss of generality, that $\bK^{-1}_\rB\bu_{\rB,h}, \bK^{-1}_\rD\bu_{\rD,h}, \f_\rB$, and $\f_\rD$,
are all piecewise polynomials. Otherwise, if $\bK_\rB, \bK_\rD, \f_\rB$, 
and $\f_\rD$ are sufficiently smooth,
one proceeds similarly to \cite[Section 6.2]{cmo2016}, so that higher order terms
given by the errors arising from suitable polynomial approximation of these functions
appear in \eqref{eq:global-efficiency}, which explains the eventual ${\tt h.o.t.}$ in 
this inequality.

\medskip
We begin the derivation of the efficiency estimates with the following result.
\begin{lem}\label{lem:efficiency-1}
There hold
\begin{equation*}
\|\div(\bu_{\rB,h})\|^2_{0,T} 
\,\leq\, 2\,|\bu_\rB - \bu_{\rB,h}|^2_{1,T} \quad \forall\,T\in\cT^\rB_h
\end{equation*}
and
\begin{equation*}
\|g_\rD - \div(\bu_{\rD,h})\|^2_{0,T} 
\,\leq\, \|\bu_\rD - \bu_{\rD,h}\|^2_{\div;T} \quad \forall\,T\in\cT^\rD_h \,.
\end{equation*}
\end{lem}
\begin{proof}
It suffices to recall that $\div(\bu_\rB)=0$ in $\Omega_\rB$ (cf. \eqref{eq:BF-model}) and $\div(\bu_\rD)=g_\rD$ in $\Omega_\rD$ (cf. \eqref{eq:Darcy-model}) and to perform simple computations. Further details are omitted.
\end{proof}

We point out that throughout this section each proof done in $2$D can be easily
extended to its three-dimensional counterpart considering $n=3$ when we apply 
\eqref{eq:inverse-inequality} in Lemma \ref{lem:inverse-inequality}. 
In that case, other positive power of the 
meshsizes $h_{T_\star}$, with $\star\in\{\rB,\rD\}$, 
will appear on the right-hand side of the efficiency estimates 
which anyway are bounded. Next, we continue providing the corresponding efficiency 
estimates of our analysis with the upper bounds for the remaining three terms 
defining $\Theta_{\rB,T}$ (cf. \eqref{eq:local-estimator-B}). 
Since the corresponding proofs are adaptations to our configuration 
of those of \cite[Lemmas 4.4, 4.5, and 4.6]{bg2010}, we only mention the main 
tools employed. 
\begin{lem}\label{lem:efficiency-2}
There exists $C>0$, independent of $h$, such that for each $T\in\cT^\rB_h$ there holds
\begin{equation}\label{eq:efficiency-2}
\begin{array}{l}
\ds h^2_T \left\|\f_\rB + \bdiv(\bsi_{\rB,h}) - \bK^{-1}_\rB\bu_{\rB,h} - \tF\,|\bu_{\rB,h}|^{\rho-2}\bu_{\rB,h} \right\|^2_{0,T} \\ [2ex]
\ds\quad \,\leq\, C\,\Big\{\|p_\rB - p_{\rB,h}\|^2_{0,T} 
+ \|\bu_\rB - \bu_{\rB,h}\|^2_{1,T} 
+ h^2_T \big\||\bu_{\rB}|^{\rho-2}\bu_{\rB} - |\bu_{\rB,h}|^{\rho-2}\bu_{\rB,h}\big\|^2_{0,T} \Big\}.
\end{array}
\end{equation}
\end{lem}
\begin{proof}
We follow an analogous reasoning to the proof of \cite[Lemma 4.4]{bg2010}.	
In fact, given $T\in \cT^\rB_h$, we define $\bchi_T := \f_\rB + \bdiv(\bsi_{\rB,h}) - \bK^{-1}_\rB\bu_{\rB,h} - \tF\,|\bu_{\rB,h}|^{\rho-2}\bu_{\rB,h}$ in $T$. Then, applying \eqref{eq:buble-property-1} to $\|\bchi_T\|_{0,T}$, we obtain
\begin{equation*}
\left\|\bchi_T\right\|^2_{0,T}
\,\leq\, c_1\,\int_T \phi_T\bchi_T\cdot ( \f_\rB + \bdiv(\bsi_{\rB,h}) - \bK^{-1}_\rB\bu_{\rB,h} - \tF\,|\bu_{\rB,h}|^{\rho-2}\bu_{\rB,h} )\,,
\end{equation*}	
where, using the identity $\f_\rB = \bK^{-1}_\rB\bu_\rB + \tF\,|\bu_{\rB}|^{\rho-2}\bu_{\rB} - \bdiv(\bsi_\rB)$ in $\Omega_\rB$ (cf. \eqref{eq:BF-model}), integration by parts and the fact that $\phi_T=0$ on $\partial\,T$ (cf. \eqref{eq:buble-property-null-boundary-T}), we deduce that
\begin{equation*}
\begin{array}{l}
\ds \left\|\bchi_T\right\|^2_{0,T}
\,\leq\, c_1\,\Bigg\{ \int_T \nabla(\phi_T\bchi_T):(\bsi_\rB - \bsi_{\rB,h}) \\[3ex]
\ds\quad \,+\, \int_T \phi_T\bchi_T\cdot\Big( \bK^{-1}_\rB(\bu_\rB-\bu_{\rB,h}) + \tF\,\big(|\bu_{\rB}|^{\rho-2}\bu_{\rB} - |\bu_{\rB,h}|^{\rho-2}\bu_{\rB,h}\big) \Big) \Bigg\} \,.
\end{array}
\end{equation*}	
Then, the Cauchy--Schwarz inequality, the inverse estimate \eqref{eq:inverse-inequality} with $\ell=1$ and $m=0$, the fact that $0\leq \phi_T \leq 1$ (cf. \eqref{eq:buble-property-null-boundary-T}), the definitions of $\bsi_\rB$ and $\bsi_{\rB,h}$ (cf. \eqref{eq:BF-model}, \eqref{eq:stress-Bh-definition}), and the triangle inequality imply that
\begin{equation}\label{eq:efficiency-2-aux}
\begin{array}{l}
\ds \left\|\bchi_T\right\|^2_{0,T}
\,\leq\, C\,\Big\{ h^{-1}_T\Big( \|p_\rB - p_{\rB,h}\|_{0,T} + |\bu_{\rB} - \bu_{\rB,h}|_{1,T} \Big) \\[2ex] 
\ds\quad \,+\, \|\bu_{\rB} - \bu_{\rB,h}\|_{0,T} + \big\||\bu_{\rB}|^{\rho-2}\bu_{\rB} - |\bu_{\rB,h}|^{\rho-2}\bu_{\rB,h}\big\|_{0,T}  \Big\} \|\bchi_T\|_{0,T} \,.
\end{array}
\end{equation}
Finally, dividing by $\|\bchi_T\|_{0,T}$ in \eqref{eq:efficiency-2-aux} and performing simple computations, we arrive at \eqref{eq:efficiency-2}, thus concluding the proof.
\end{proof}

\begin{lem}\label{lem:efficiency-3}
There exists $C>0$, independent of $h$, such that for each $e\in\cE_h(\Omega_\rB)$ there holds
\begin{equation}\label{eq:efficiency-3}
\begin{array}{l}
\ds h_e\big\|\jump{\bsi_{\rB,h}\bn}\big\|^2_{0,e}  
\leq C\sum_{T\subseteq \omega_e} \Big\{ \|p_\rB - p_{\rB,h}\|^2_{0,T} \\[3ex] 
\ds\quad +\, \|\bu_\rB - \bu_{\rB,h}\|^2_{1,T}  
+ h^2_T \big\||\bu_{\rB}|^{\rho-2}\bu_{\rB} - |\bu_{\rB,h}|^{\rho-2}\bu_{\rB,h}\big\|^2_{0,T} \Big\} \,,
\end{array}
\end{equation}
where $\omega_e$ is the union of the two triangles in $\cT^\rB_h$ having $e$ as an edge.
\end{lem}
\begin{proof}
First, we proceed as in the proof of \cite[(4.11) in Lemma 4.5]{bg2010} by noting that $\bsi_\rB\in \bbH(\bdiv;\Omega_\rB)$ and thus $\jump{\bsi_\rB\bn} = \0$ on each $e\in \cE_h(\Omega_\rB)$. 
Then, by employing \eqref{eq:buble-property-2}, integration by parts formula on each $T\subseteq \omega_e$, the inverse estimate \eqref{eq:inverse-inequality} with $\ell=1$ and $m=0$, the fact that $\bdiv(\bsi_\rB) = - \f_\rB + \bK^{-1}_\rB\bu_\rB + \tF\,|\bu_{\rB}|^{\rho-2}\bu_{\rB}$ in $\Omega_\rB$ (cf. \eqref{eq:BF-model}), and the estimate \eqref{eq:buble-property-3}, we obtain
\begin{equation*}
\begin{array}{l}
\ds h_e\big\|\jump{\bsi_{\rB,h}\bn}\big\|^2_{0,e}  
\leq C\,\sum_{T\subseteq \omega_e} \Bigg\{ 
h^2_T \left\|\f_\rB + \bdiv(\bsi_{\rB,h}) - \bK^{-1}_\rB\bu_{\rB,h} - \tF\,|\bu_{\rB,h}|^{\rho-2}\bu_{\rB,h} \right\|^2_{0,T} \\[2ex]
\ds\quad \,+\, h^2_T\,\|\bu_{\rB} - \bu_{\rB,h}\|^2_{0,T} 
+ h^2_T\,\big\||\bu_{\rB}|^{\rho-2}\bu_{\rB} - |\bu_{\rB,h}|^{\rho-2}\bu_{\rB,h}\big\|^2_{0,T}
+ \|\bsi_\rB - \bsi_{\rB,h}\|^2_{0,T} 
\Bigg\} \,,
\end{array}
\end{equation*}
which, together with \eqref{eq:efficiency-2} and the definitions of $\bsi_\rB$ and $\bsi_{\rB,h}$ (cf. \eqref{eq:BF-model}, \eqref{eq:stress-Bh-definition}), yields \eqref{eq:efficiency-3} and complete the proof.
\end{proof}

\begin{lem}\label{lem:efficiency-4}
There exists $C>0$, independent of $h$, such that for each $e\in\cE_h(\Sigma)$ there holds
\begin{equation}\label{eq:efficiency-4}
\begin{array}{l}
\ds h_e\,\left\|\bsi_{\rB,h}\bn + \lambda_h\bn \right\|^2_{0,e}
\leq C\,\Big\{ \|p_\rB - p_{\rB,h}\|^2_{0,T} + \|\bu_\rB - \bu_{\rB,h}\|^2_{1,T} \\ [3ex]
\ds\quad +\,\, h^2_T\,\big\||\bu_{\rB}|^{\rho-2}\bu_{\rB} - |\bu_{\rB,h}|^{\rho-2}\bu_{\rB,h}\big\|^2_{0,T} 
\,+\, h_e\,\|\lambda - \lambda_h\|^2_{0,e}\Big\} \,,
\end{array}
\end{equation}
where $T$ is the triangle of $\cT^\rB_h$ having $e$ as an edge.
\end{lem}
\begin{proof}
First, following the approach in the proof of \cite[Lemma 4.6]{bg2010}, we define $\bchi_e:= \bsi_{\rB,h}\bn + \lambda_h\bn$ on $e$. 
Then, applying \eqref{eq:buble-property-2}, using that $\bsi_\rB\bn + \lambda\bn = \0$ on $\Sigma$ (cf. \eqref{eq:transmission-condition}), recalling that $\phi_e=0$ on $\partial\,T\setminus e$, with $T$ being the triangle on $\cT^\rB_h$ having $e$ as an edge, and integrating by parts on $T$, we obtain that
\begin{equation*}
\begin{array}{l}
\ds \big\|\bchi_e\big\|^2_{0,e}  
\leq c_2 \int_T \nabla(\phi_e\bL(\bchi_e)):(\bsi_{\rB,h} - \bsi_\rB) \\[3ex]
\ds\quad +\, c_2\int_T \phi_e\,\bL(\bchi_e)\cdot\bdiv(\bsi_{\rB,h} - \bsi_\rB)
\,+\, c_2\int_e \phi_e\bchi_e\cdot (\lambda_h - \lambda)\bn  \,.
\end{array}
\end{equation*}
Hence, the Cauchy--Schwarz inequality, the inverse estimate \eqref{eq:inverse-inequality} with $\ell=1$ and $m=0$, the fact that $0\leq \phi_e \leq 1$ (cf. \eqref{eq:buble-property-null-boundary-e}), the fact that $\bdiv(\bsi_\rB) = - \f_\rB + \bK^{-1}_\rB\bu_\rB + \tF\,|\bu_{\rB}|^{\rho-2}\bu_{\rB}$ in $\Omega_\rB$ (cf. \eqref{eq:BF-model}), and adding and subtracting suitable terms, imply from the above equation that
\begin{equation*}
\begin{array}{l}
\ds \big\|\bchi_e\big\|^2_{0,e}  
\leq C\,\Big\{ h^{-1}_T\,\|\bsi_\rB - \bsi_{\rB,h}\|_{0,T} 
+ \|\f_\rB + \bdiv(\bsi_{\rB,h}) - \bK^{-1}_\rB\bu_{\rB,h} - \tF\,|\bu_{\rB,h}|^{\rho-2}\bu_{\rB,h}\|_{0,T} \\[2ex]
\ds\quad +\, \|\bu_{\rB} - \bu_{\rB,h}\|_{0,T} 
+ \big\||\bu_{\rB}|^{\rho-2}\bu_{\rB} - |\bu_{\rB,h}|^{\rho-2}\bu_{\rB,h}\big\|_{0,T} \Big\}\|\phi_e\bL(\bchi_e)\|_{0,T}
+ C\,\|\lambda - \lambda_h\|_{0,e}\,\|\bchi_e\|_{0,e}  \,.
\end{array}
\end{equation*}
Now, applying the estimate \eqref{eq:buble-property-3}, we see that $\|\phi_e\bL(\bchi_e)\|_{0,T} \leq c_3\,h^{1/2}_e\|\bchi_e\|_{0,e}$, which together with the fact that $h_e\leq h_T$, yields
\begin{equation*}
\begin{array}{l}
\ds h_e\,\big\|\bchi_e\big\|^2_{0,e}  
\leq C\,\Big\{ \|\bsi_\rB - \bsi_{\rB,h}\|^2_{0,T} 
+ h^2_T\,\|\f_\rB + \bdiv(\bsi_{\rB,h}) - \bK^{-1}_\rB\bu_{\rB,h} - \tF\,|\bu_{\rB,h}|^{\rho-2}\bu_{\rB,h}\|^2_{0,T} \\[2ex]
\ds\quad +\, h^2_T\,\|\bu_{\rB} - \bu_{\rB,h}\|^2_{0,T} 
+ h^2_T\,\big\||\bu_{\rB}|^{\rho-2}\bu_{\rB} - |\bu_{\rB,h}|^{\rho-2}\bu_{\rB,h}\big\|^2_{0,T}
+ h_e\,\|\lambda - \lambda_h\|^2_{0,e} \Big\} \,,
\end{array}
\end{equation*}
which, together with \eqref{eq:efficiency-2} and the definitions of $\bsi_\rB$ and $\bsi_{\rB,h}$ (cf. \eqref{eq:BF-model}, \eqref{eq:stress-Bh-definition}), yields \eqref{eq:efficiency-4} and complete the proof.
\end{proof}

At this point, we stress that the local efficiency estimates for the remaining terms defining
$\Theta_\rD$ (cf. \eqref{eq:local-estimator-D}) have already been proved in the literature 
by using the localization technique based on triangle-bubble and edge-bubble functions 
(cf. eqs. \eqref{eq:buble-property-null-boundary-T}--\eqref{eq:buble-property-null-boundary-e} and Lemma \ref{lem:properties-bubble}), the local inverse inequality 
(cf. \eqref{eq:inverse-inequality}) with $\ell=1$ and $m=0$, and the discrete trace inequality 
(cf. \eqref{eq:discrete-trace-inequality}). More precisely, we provide the following result.
\begin{lem}\label{lem:efficiency-5}
There exist positive constants $C_i, i\in \{1,\dots,6\}$, independent of $h$, such that
\begin{itemize}
\item[{\rm (a)}] $\ds h^2_T\big\|\rot\left(\f_\rD - \bK^{-1}_\rD\bu_{\rD,h}\right)\big\|^2_{0,T}
\,\leq\, C_1\,\|\bu_\rD - \bu_{\rD,h}\|_{0,T} \quad \forall\,T\in \cT^\rD_h$\,,

\item[{\rm (b)}] $\ds h_e\,\left\|\jump{\left(\f_\rD - \bK^{-1}_\rD\bu_{\rD,h}\right)\cdot\bt}\right\|^2_{0,e}
\,\leq\, C_2\,\|\bu_\rD - \bu_{\rD,h}\|_{0,\omega_e}$ for all $e\in \cE_h(\Omega_\rD)$, where the set $\omega_e$ is given by $\omega_e:=\cup \big\{ T'\in \cT^\rD_h : \quad e\in \cE(T') \big\}$\,,

\item[{\rm (c)}] $\ds h^2_T \left\|\f_\rD - \bK^{-1}_\rD \bu_{\rD,h}\right\|^2_{0,T} 
\,\leq\, C_3\,\Big\{ \|p_\rD - p_{\rD,h}\|^2_{0,T} + h^2_T\,\|\bu_\rD - \bu_{\rD,h}\|^2_{0,T} \Big\} \quad \forall\,T\in \cT^\rD_h$\,,

\item[{\rm (d)}] $\ds h_e\,\|\bu_{\rB,h}\cdot\bn - \bu_{\rD,h}\cdot\bn\|^2_{0,e}
\,\leq\, C_4\,\Big\{ \|\bu_\rB - \bu_{\rB,h}\|^2_{0,T_\rB} + h^2_{T_\rB}\,|\bu_\rB - \bu_{\rB,h}|^2_{1,T_\rB} + \|\bu_\rD - \bu_{\rD,h}\|^2_{0,T_\rD}  + h^2_{T_\rD}\,\|\div(\bu_\rD - \bu_{\rD,h})\|^2_{0,T_\rD} \Big\}$, where $T_\star$ are the triangles of $\cT^\star_h$, with $\star\in \{\rB,\rD\}$, having $e$ as an edge\,,

\item[{\rm (e)}] $\ds h_e\,\|\lambda_h - p_{\rD,h}\|^2_{0,e}
\,\leq\, C_5\,\Big\{ \|p_\rD - p_{\rD,h}\|^2_{0,T} + h^2_T\,\|\bu_\rD - \bu_{\rD,h}\|^2_{0,T}
+ h_e\,\|\lambda - \lambda_h\|^2_{0,e} \Big\}$ for all $e\in \cE_h(\Omega_\rD)$, where $T$ is the triangle of $\cT^\rD_h$ having $e$ as an edge\,,

\item[{\rm (f)}] $\ds \sum_{e\in \cE_h(\Sigma)} h_e\,\left\|\left(\f_\rD - \bK^{-1}_\rD\bu_{\rD,h}\right)\cdot\bt - \frac{d\,\lambda_h}{d\,\bt} \right\|^2_{0,e}
\,\leq\,C_6\,\Big\{ \|\bu_\rD - \bu_{\rD,h}\|^2_{0,\Omega_\rD} \,+\, \|\lambda - \lambda_h\|^2_{1/2,\Sigma} \Big\}$\,.
\end{itemize}
\end{lem}
\begin{proof}
The estimates ${\rm (a)}, {\rm (b)}$, and ${\rm (c)}$ follow straightforwardly from a slight modification of the proof of \cite[Lemmas 6.1, 6.2, and 6.3]{c1997}, respectively, whereas ${\rm (d)}$ and ${\rm (e)}$ follow from \cite[Lemmas 4.7 and 4.12]{bg2010}, respectively.
Finally, the proof of ${\rm (f)}$ relies on the facts that $\nabla p_\rD = \f_\rD - \bK^{-1}_\rD\,\bu_\rD$ in $\Omega_\rD$ (cf. \eqref{eq:Darcy-model}) and $\lambda=p_\rD$ on $\Sigma$.
We omit further details and refer the reader to \cite[Lemma 5.7]{g2004}.
\end{proof}

In order to complete the global efficiency given by \eqref{eq:global-efficiency} (cf. Theorem \ref{th:efficiency}), we now need to estimate the terms $\left\||\bu_\rB|^{\rho-2}\bu_\rB - |\bu_{\rB,h}|^{\rho-2}\bu_{\rB,h}\right\|^{2}_{0,T}$ and $h_e\|\lambda - \lambda_h\|^2_{0,e}$ appearing in the upper bounds provided by Lemmas \ref{lem:efficiency-2}, \ref{lem:efficiency-3}, \ref{lem:efficiency-4}, and \ref{lem:efficiency-5}-(e), respectively. 
In particular, the term $h_e\|\lambda - \lambda_h\|^2_{0,e}$ 
appearing in Lemmas \ref{lem:efficiency-4} and \ref{lem:efficiency-5}-(e), is bounded as follows:
\begin{equation}\label{eq:lambda-bound-H-1-2}
\sum_{e\in\cE_h(\Sigma)} h_e\,\|\lambda - \lambda_h\|^2_{0,e} \,\le\, 
h \, \|\lambda - \lambda_h\|^2_{0,\Sigma} \,\le\, C\,h\,\|\lambda - \lambda_h\|^2_{1/2,\Sigma}\,.
\end{equation}

On the other hand, to bound $\left\||\bu_\rB|^{\rho-2}\bu_\rB - |\bu_{\rB,h}|^{\rho-2}\bu_{\rB,h}\right\|^{2}_{0,T}$,
we proceed as in \cite[eqs. (5.21)--(5.22)]{ce2024} and recall from \cite[Lemma 5.3]{gm1975} that for each $m \ge 2$, there 
exists a constant $C(m) > 0$ such that
\begin{equation}\label{eq:Glowinsky-inequality}
\big| |\bz|^{m-2}\bz - |\by|^{m-2}\by \big|
\,\le\, C(m) \,\big(|\bz| + |\by|\big)^{m-2} |\bz - \by| \quad \forall\, \bz, \, \by \,\in\, \R^2 \,.
\end{equation}
Thus, applying \eqref{eq:Glowinsky-inequality} with $m=\rho$, the H\"older inequality with $p=3/2$ and $q=3$ satisfying $1/p + 1/q=1$,
and simple algebraic manipulations, to obtain
\begin{equation*}
\left\||\bu_\rB|^{\rho-2}\bu_\rB - |\bu_{\rB,h}|^{\rho-2}\bu_{\rB,h}\right\|^{2}_{0,T} 
\,\leq\, c_\rho\,\Big( \|\bu_\rB\|_{0,3(\rho-2);T}^{2(\rho-2)} + \|\bu_{\rB,h}\|_{0,3(\rho-2);T}^{2(\rho-2)}\Big) \|\bu_\rB - \bu_{\rB,h}\|_{0,6;T}^2 \,,
\end{equation*}
with $c_\rho := 2^{2\,\rho - 5}\,C^2(\rho)$.
Then, applying H\"older inequality and some algebraic computations, we find that
\begin{equation}\label{eq:Forch_bound-2}
\begin{array}{l}
\ds \sum_{T\in \cT_h} \left\||\bu_\rB|^{\rho-2}\bu_\rB - |\bu_{\rB,h}|^{\rho-2}\bu_{\rB,h}\right\|^{2}_{0,T} \\[2ex]
\ds\quad \,\leq\, c_\rho\,\left\{ \sum_{T\in \cT^\rB_h} \Big( \|\bu_\rB\|_{0,3(\rho-2);T}^{2(\rho-2)} 
+ \|\bu_{\rB,h}\|_{0,3(\rho-2);T}^{2(\rho-2)}\Big)^{3/2} \right\}^{2/3} \left\{ \sum_{T\in \cT^\rB_h} \|\bu_\rB - \bu_{\rB,h}\|_{0,6;T}^6 \right\}^{1/3} \\[4ex]
\ds\quad 
\,\leq\, \sqrt[3]{2}\,c_\rho\, \Big( \|\bu_\rB\|_{0,3(\rho-2);\Omega_\rB}^{2(\rho-2)} + \|\bu_{\rB,h}\|_{0,3(\rho-2);\Omega_\rB}^{2(\rho-2)}\Big) \|\bu_\rB - \bu_{\rB,h}\|_{0,6;\Omega_\rB}^2 \,.
\end{array}
\end{equation}
Then, using the continuous injections $\bi_6: \bH^1(\Omega_\rB)\to \bL^6(\Omega_\rB)$
and $\bi_{3(\rho-2)}:\bH^1(\Omega_\rB)\to \bL^{3(\rho-2)}(\Omega_\rB)$, with $3(\rho-2)\in [3,6]$ (cf. \eqref{eq:Sobolev-inequality}), and 
the fact that $\|\bu_\rB\|_{1,\Omega_\rB}$ and $\|\bu_{\rB,h}\|_{1,\Omega_\rB}$ are bounded by data (cf. \eqref{eq:a-priori-bound-u}, \eqref{eq:a-priori-bound-uh}),
we deduce from \eqref{eq:Forch_bound-2} that there exists a constant $C>0$, 
depending only on data and other constants, and hence independent of $h$, such that
\begin{equation}\label{eq:Forch_bound-3}
\sum_{T\in\cT^\rB_h} \big\| |\bu_\rB|^{\rho-2}\bu_\rB - |\bu_{\rB,h}|^{\rho-2}\bu_{\rB,h} \big\|_{0,T}^2 
\,\leq\, C\,\|\bu_\rB - \bu_{\rB,h}\|_{1,\Omega_\rB}^2 \,.
\end{equation}

Consequently, it is not difficult to see that \eqref{eq:global-efficiency} follows 
from the definition of $\Theta_{\tBFD}$ (cf. \eqref{eq:global-estimator}),
Lemmas \ref{lem:efficiency-1}, \ref{lem:efficiency-2}, \ref{lem:efficiency-3}, and \ref{lem:efficiency-4}, and the estimates \eqref{eq:lambda-bound-H-1-2} and  \eqref{eq:Forch_bound-3}.


\section{Numerical results}\label{sec:numerical-results}

This section serves to illustrate the performance and accuracy of the mixed finite element 
scheme \eqref{eq:discrete-mixed-formulation} along with the reliability and efficiency 
properties of the {\it a posteriori} error estimator $\Theta_{\tBFD}$ (cf. \eqref{eq:global-estimator}) 
derived in Section \ref{sec:a-posteriori-error-analysis}. 
The implementation is based on a {\tt FreeFEM} code \cite{Hecht2012}. Regarding the 
implementation of the Newton iterative method associated to \eqref{eq:discrete-mixed-formulation} 
(see \cite[eq. (5.1) in Section 5]{cd2023} for details), the iterations are 
terminated once the relative error of the entire coefficient vectors between two consecutive 
iterates, say $\coeff^{m}$ and $\coeff^{m+1}$, is sufficiently small, i.e.,
\begin{equation*}
\frac{\|\coeff^{m+1} - \coeff^m\|_{\ell^2}}{\|\coeff^{m+1}\|_{\ell^2}} \,\leq\, \tol,
\end{equation*}
where $\|\cdot\|_{\ell^2}$ is the standard $\ell^2$-norm in $\R^{\DOF}$, 
with $\DOF$ denoting the total number of degrees of
freedom defining the finite element subspaces $\bH_h$ and $\bQ_h$ (cf. \eqref{eq:FEM-1}--\eqref{eq:FEM-2}), 
and $\tol$ is a fixed tolerance chosen as $\tol=1\textup{E}-06$.

The global error and the effectivity index associated to the global estimator
$\Theta_{\tBFD}$ are denoted, respectively, by
\begin{equation*}
\re(\vec{\bsi}) \,:=\, \re(\bu_\rB) + \re(\bu_\rD) + \re(p_\rB) + \re(p_\rD) + \re(\lambda) \qan
\eff(\Theta_{\tBFD}) \,:=\, \frac{\re(\vec{\bsi})}{\Theta_{\tBFD}} \,,
\end{equation*}
where
\begin{equation*}
\begin{array}{c}
\re(\bu_\rB) \,:=\, \|\bu_\rB - \bu_{\rB,h}\|_{1,\Omega_\rB}\,,\quad 
\re(\bu_\rD) \,:=\, \|\bu_\rD - \bu_{\rD,h}\|_{\div;\Omega_\rD} \,, \\ [2ex]
\re(p_\rB) \,:=\, \|p_\rB - p_{\rB,h}\|_{0,\Omega_\rB}\,,\quad
\re(p_\rD) \,:=\, \|p_\rD - p_{\rD,h}\|_{0,\Omega_\rD}\,,\quad 
\re(\lambda) \,:=\, \|\lambda - \lambda_h\|_{1/2,\Sigma} \,.
\end{array}
\end{equation*}
Notice that, for ease of computation, the interface norm
$\|\lambda - \lambda_h\|_{1/2,\Sigma}$ will be replaced by $\|\lambda - \lambda_h\|_{(0,1),\Sigma}$ with
\begin{equation*}
\|\xi\|_{(0,1),\Sigma} \,:=\, \|\xi\|^{1/2}_{0,\Sigma} \, \|\xi\|^{1/2}_{1,\Sigma} 
\quad \forall\,\xi\in \H^1(\Sigma) \,,
\end{equation*}
owing to the fact that $\H^{1/2}(\Sigma)$ is the interpolation space with index $1/2$ between $\H^1(\Sigma)$ and $\L^2(\Sigma)$.

Moreover, using the fact that $\DOF^{-1/2}\cong h$, the respective experimental 
rates of convergence are computed as
\begin{equation*}
\sr(\diamond) \,:=\, -2\,\frac{\log(\re(\diamond)/\re'(\diamond))}{\log(\DOF/\DOF')} \quad 
\mbox{for each } \diamond\in\Big\{ \bu_\rB, \bu_\rD, p_\rB, p_\rD, \lambda, \vec{\bsi} \Big\}\,,
\end{equation*}
where $\DOF$ and $\DOF'$ denote the total degrees of freedom associated to 
two consecutive triangulations with errors $e$ and $e'$, respectively. 

For each example shown below we take $\bu^0_{\rB,h}=(0.1,0)^\rt$ as initial guess.
In addition, the condition $(p_h,1)_{\Omega} = 0$ is imposed via a penalization strategy.

Example 1 is used to show the accuracy of the method and the behavior of the effectivity index of the {\it a posteriori} error estimator $\Theta_{\tBFD}$, whereas Examples 2 and 3 are utilized 
to illustrate the associated adaptive algorithm, with and without manufactured solution,
respectively.
The corresponding adaptivity procedure, taken from \cite{Verfurth}, is described as follows:
\begin{enumerate}
\item[(1)] Start with a coarse mesh $\cT_h := \cT^\rB_h\cup\cT^\rD_h$.
\item[(2)] Solve the Newton iterative method \cite[eq. (5.1) in Section 5]{cd2023} for the current mesh $\cT_h$.
\item[(3)] Compute the local indicator $\Theta_{\tBFD,T}$ for each $T\in\cT_h:=\cT^\rB_h\cup \cT^\rD_h$, where 
\begin{equation*}
\Theta_{\tBFD,T} \,:=\, \Big\{ \Theta^2_{\rB,T} + \Theta^2_{\rD,T} \Big\}^{1/2} \,,\quad \text{ (cf. \eqref{eq:local-estimator-B}, \eqref{eq:local-estimator-D})}
\end{equation*}
\item[(4)] Check the stopping criterion and decide whether to finish or go to next step.
\item[(5)] 
Use the automatic meshing algorithm {\tt adaptmesh} from \cite[Section 9.1.9]{Hecht2018} to refine each $T'\in \cT_h$ satisfying:
\begin{equation*}
\Theta_{\tBFD,T'} \,\geq\, C_{\sf adt}\,\frac{1}{\#\,T} \sum_{T\in \cT_h} \Theta_{\tBFD,T},\quad \mbox{for some }\, C_{\sf adt}\in (0,1),
\end{equation*}
where $\#\,T$ denotes the number of triangles of the mesh $\cT_h$.
In particular, in Examples 2 and 3 below we take $C_{\sf adt} = 0.8$. 
	
\item[(6)] Define resulting meshes as current meshes $\cT^\rB_h$ and $\cT^\rD_h$, and go to step (2).
\end{enumerate}

\subsection*{Example 1: Accuracy assessment with a smooth solution in a rectangular domain.}

In the first example, we consider the rectangular domain $\Omega=\Omega_\rB\cup\Sigma\cup \Omega_\rD$, where
\begin{equation*}
\Omega_\rB := (0,1)\times(1,2)\,,\quad \Omega_\rD := (0,1)^2\,,\qan \Sigma := (0,1)\times\{1\} \,.
\end{equation*} 
We consider the model parameter $\rho=3$, $\mu=1$, $\tF=10$, $\bK_\rB=\bbI$, $\bK_\rD=0.5\times\bbI$,
and the data $\f_\rB, \f_\rD$, and $g_\rD$ are chosen so that the exact solution in the rectangular 
domain $\Omega = \Omega_\rB\cup\Sigma\cup\Omega_\rD$ is given by the smooth functions
\begin{equation*}
\begin{array}{c}
\bu_\rB(x_1,x_2) := \left(\begin{array}{r}
-\sin(\pi x_1)\cos(\pi x_2) \\ \cos(\pi x_1)\sin(\pi x_2)
\end{array}\right),\quad 
\bu_\rD(x_1,x_2) := \left(\begin{array}{r}
\sin(\pi x_1)\,\exp(x_2) \\ \exp(x_1)\,\sin(\pi x_2)  
\end{array}\right),\\ [3ex]
p_\star(x_1,x_2) := x_1\cos(\pi x_2) \qin \Omega_\star,\quad \mbox{ with } \star\in\{\rB,\rD\}.
\end{array}
\end{equation*}
Notice that this solution satisfies $\bu_\rB\cdot\bn = \bu_\rD\cdot\bn$ on $\Sigma$. 
However, the second transmission condition in \eqref{eq:transmission-condition} is not satisfied, 
whereas the Dirichlet boundary condition for the Brinkman--Forchheimer velocity on $\Gamma_\rB$ and 
the Neumann boundary condition for the Darcy velocity on $\Gamma_\rD$ are both 
non-homogeneous. 
This introduces additional contributions that are incorporated into the right-hand side of the resulting system and into the {\it a posteriori} error estimator $\Theta_{\tBFD}$ (cf. \eqref{eq:global-estimator}).
The errors and associated rates of convergence are reported in Table \ref{table1-example1},
which align with the theoretical bounds established in \cite[Theorem 4.6]{cd2023}.
Additionally, we compute the global {\it a posteriori} error indicator $\Theta_{\tBFD}$ 
and assess its reliability and efficiency using the effectivity index. 
Note that the estimator decreases in such a way that the effectivity index remains consistently bounded.

\begin{table}[ht]
\begin{center}
\small{		
\begin{tabular}{r|c||c|c|c|c||c||c|c|c|c}
\hline
$\DOF$  & $h_\rB$ & $\re(\bu_\rB)$ & $\sr(\bu_\rB)$ & $\re(p_\rB)$ & $\sr(p_\rB)$ & $h_\rD$ & $\re(\bu_\rD)$ & $\sr(\bu_\rD)$ & $\re(p_\rD)$ & $\sr(p_\rD)$ \\  \hline \hline
258    & 0.373 & 5.6E-01 &  --   & 2.5E-01 &   --  & 0.373 & 1.2E-00 &   --  & 1.0E-01 &   --  \\ 
1016   & 0.196 & 2.6E-01 & 1.120 & 8.3E-02 & 1.627 & 0.196 & 5.5E-01 & 1.115 & 4.2E-02 & 1.278 \\ 
3784   & 0.103 & 1.3E-01 & 1.071 & 3.6E-02 & 1.279 & 0.103 & 2.7E-01 & 1.072 & 1.9E-02 & 1.189 \\ 
14868  & 0.057 & 6.6E-02 & 0.972 & 1.7E-02 & 1.080 & 0.051 & 1.4E-01 & 0.970 & 9.9E-03 & 0.984 \\ 
58822  & 0.027 & 3.2E-02 & 1.044 & 8.9E-03 & 0.935 & 0.027 & 6.9E-02 & 1.015 & 4.9E-03 & 1.012 \\ 
235922 & 0.014 & 1.6E-02 & 1.001 & 4.2E-03 & 1.071 & 0.014 & 3.5E-02 & 1.004 & 2.5E-03 & 1.002 \\ 
\hline 
\end{tabular}
			
\medskip
	
\begin{tabular}{c||c|c||c|c||c|c||c}
\hline
$h_\Sigma$ & $\re(\lambda)$ & $\sr(\lambda)$ & $\re(\vec{\bsi})$ & $\sr(\vec{\bsi})$ & $\Theta_{\tBFD}$ & $\eff(\Theta_{\tBFD})$ & iter \\  \hline \hline
1/2  & 1.4E-01 &   --  & 1.3E-00 &   --  & 5.3E-00 & 0.251 & 5 \\ 
1/4  & 2.8E-02 & 2.367 & 6.1E-01 & 1.137 & 2.5E-00 & 0.248 & 5 \\ 
1/8  & 7.5E-03 & 1.987 & 3.0E-01 & 1.077 & 1.2E-00 & 0.243 & 5 \\ 
1/16 & 2.1E-03 & 1.891 & 1.6E-01 & 0.972 & 6.2E-01 & 0.251 & 5 \\ 
1/32 & 6.2E-04 & 1.751 & 7.7E-02 & 1.019 & 3.1E-01 & 0.248 & 5 \\ 
1/64 & 1.4E-04 & 2.171 & 3.8E-02 & 1.004 & 1.6E-01 & 0.247 & 5 \\ 
\hline 
\end{tabular}
\caption{[{\sc Example 1}] $\bBR - \bRT_0 - \rP_0 - \rP_0$ primal-mixed scheme with quasi-uniform refinement.}
\label{table1-example1}
}
\end{center}
\end{table}

\subsection*{Example 2: Adaptivity in a 2D helmet-shaped domain.}

In the second example, similarly to \cite[Example 2, Section 6]{cgos2021}, 
we consider a 2D helmet-shaped domain. Specifically, the domain is defined 
as $\Omega = \Omega_\rB \cup \Sigma \cup \Omega_\rD$, where
\begin{equation*}
\Omega_\rB := (-1,1)\times(0,1.25)\backslash (-0.75,0.75)\times(0.25,1.25)\,,\quad
\Omega_\rD := (-1,1)\times(-0.5,0) \,,
\end{equation*}
and $\Sigma := (-1,1)\times \{0\}$ (see the first plot of Figure \ref{figure:Ex2-domain-solutions} below).
We use the model parameters $\rho=3.5$, $\mu=1$, $\tF=10$, $\bK_{\rB}=\bbI$, 
and $\bK_{\rD}=10^{-1}\times\bbI$. The data $\f_\rB, \f_\rD$, and $g_\rD$ are adjusted so that the 
exact solution in the $2$D helmet-shaped domain $\Omega$ is given by the smooth functions
\begin{equation*}
\begin{array}{c}
\ds \bu_\rB(x_1,x_2) := \left(\begin{array}{r}
\dfrac{(x_2 - 0.26)}{r_1(x_1,x_2)} + \dfrac{(x_2 - 0.26)}{r_2(x_1,x_2)} \\ [3ex] 
-\dfrac{(x_1 + 0.74)}{r_1(x_1,x_2)} - \dfrac{(x_1 - 0.74)}{r_2(x_1,x_2)}
\end{array}\right),\quad
\bu_\rD(x_1,x_2) := \left(\begin{array}{r}
\sin(\pi x_1)\,x_2 \\ x_1\,\sin(\pi x_2)  
\end{array}\right),\\ [7ex]
\ds \,\,p_\star(x_1,x_2) := \sin(\pi x_1)\,x_2 \qin \Omega_\star,\quad \mbox{ with } \star\in \big\{\rB,\rD\big\},
\end{array}
\end{equation*}
where 
\begin{equation*}
r_1(x_1,x_2) := \sqrt{(x_1 + 0.74)^2 + (x_2 - 0.26)^2} \qan 
r_2(x_1,x_2) := \sqrt{(x_1 - 0.74)^2 + (x_2 - 0.26)^2}.
\end{equation*}
Tables \ref{table:ex2-uniform} and \ref{table:ex2-adaptive} along with 
Figure \ref{fig:Ex2-loglog-errors}, summarize the convergence history of the method applied to a sequence of quasi-uniform and adaptively refined triangulations of the domain. Suboptimal rates are observed in the first case, whereas adaptive refinement based on the {\it a posteriori} error indicator $\Theta_\tBFD$ (cf. \eqref{eq:global-estimator}) yields optimal convergence and stable effectivity indexes. Notice how the adaptive algorithms improve the method's efficiency by providing high-quality solutions at a lower computational cost. For instance, it is possible to achieve a better result (in terms of $\re(\vec{\bsi})$) with approximately only $2.6\%$ of the degrees of freedom compared to the last quasi-uniform mesh for the mixed scheme. Furthermore, Figure \ref{figure:Ex2-domain-solutions} shows the domain configuration in the initial mesh, the second component of the velocity, and the pressure field over the entire domain, computed using the adaptive $\bBR - \bRT_0 - \rP_0 - \rP_0$ primal-mixed scheme (via $\Theta_\tBFD$) with $374,444$ degrees of freedom and $116,018$ triangles. We observe that the second component of the velocity exhibits high gradients near the vertices $(-0.75,0.25)$ and $(0.75,0.25)$. Examples of some adapted meshes are shown in Figure \ref{figure:Ex2-adapted-meshes}, where a clear clustering of elements near the vertices in $\Omega_\rB$ of the 2D helmet-shaped domain is observed, as expected.

\begin{table}[ht]
\begin{center}
\small{		
\begin{tabular}{r|c||c|c|c|c||c||c|c|c|c}
\hline
$\DOF$  & $h_\rB$ & $\re(\bu_\rB)$ & $\sr(\bu_\rB)$ & $\re(p_\rB)$ & $\sr(p_\rB)$ & $h_\rD$ & $\re(\bu_\rD)$ & $\sr(\bu_\rD)$ & $\re(p_\rD)$ & $\sr(p_\rD)$ \\  \hline \hline
918    & 0.188 & 2.9E-00 &   --  & 1.4E-00 &   --  & 0.200 & 1.6E-01 &   --  & 3.7E-02 &   --  \\ 
3686   & 0.100 & 2.1E-00 & 0.452 & 7.3E-01 & 0.911 & 0.095 & 8.1E-02 & 1.017 & 1.4E-02 & 1.359 \\ 
13670  & 0.050 & 1.6E-00 & 0.477 & 4.6E-01 & 0.710 & 0.049 & 4.1E-02 & 1.024 & 7.0E-03 & 1.076 \\ 
54755  & 0.026 & 8.9E-01 & 0.801 & 3.1E-01 & 0.568 & 0.026 & 2.0E-02 & 1.018 & 3.5E-03 & 0.990 \\ 
213836 & 0.014 & 5.2E-01 & 0.776 & 1.6E-01 & 0.994 & 0.015 & 1.0E-02 & 1.024 & 1.7E-03 & 1.071 \\ 
857896 & 0.007 & 2.7E-01 & 0.931 & 8.7E-02 & 0.848 & 0.007 & 5.1E-03 & 0.986 & 8.5E-04 & 0.994 \\ 
\hline 
\end{tabular}

\medskip

\begin{tabular}{c||c|c||c|c||c|c||c}
\hline
$h_\Sigma$ & $\re(\lambda)$ & $\sr(\lambda)$ & $\re(\vec{\bsi})$ & $\sr(\vec{\bsi})$ & $\Theta_{\tBFD}$ & $\eff(\Theta_{\tBFD})$ & iter \\  \hline \hline
1/4   & 1.2E-01 &   --  & 3.2E-00 &   --  & 2.5E+01 & 0.129 & 5 \\ 
1/8   & 3.3E-02 & 1.847 & 2.2E-00 & 0.518 & 1.9E+01 & 0.120 & 5 \\ 
1/16  & 1.1E-02 & 1.760 & 1.6E-00 & 0.499 & 1.0E+01 & 0.162 & 5 \\ 
1/32  & 5.9E-03 & 0.831 & 9.4E-01 & 0.779 & 5.6E-00 & 0.168 & 5 \\ 
1/64  & 1.2E-03 & 2.288 & 5.5E-01 & 0.797 & 3.4E-00 & 0.161 & 5 \\ 
1/128 & 5.5E-04 & 1.185 & 2.9E-01 & 0.924 & 1.7E-00 & 0.166 & 5 \\ 
\hline 
\end{tabular}
\caption{[{\sc Example 2}] $\bBR - \bRT_0 - \rP_0 - \rP_0$ primal-mixed scheme with quasi-uniform refinement.}\label{table:ex2-uniform}
}
\end{center}
\end{table}
%
\begin{table}[ht]
\begin{center}
\small{	
\begin{tabular}{r||c|c|c|c|c|c|c|c}
\hline
$\DOF$  &  $\re(\bu_\rB)$ & $\sr(\bu_\rB)$ & $\re(p_\rB)$ & $\sr(p_\rB)$ & $\re(\bu_\rD)$ & $\sr(\bu_\rD)$ & $\re(p_\rD)$ & $\sr(p_\rD)$ \\  \hline \hline
918    & 2.9E-00 &   --  & 1.4E-00 &  --   & 1.64E-01 &   --  & 3.7E-02 &  --   \\ 
1555   & 1.6E-00 & 2.239 & 2.7E-00 &  --   & 1.63E-01 & 0.022 & 2.6E-00 &  --   \\ 
2568   & 9.3E-01 & 2.172 & 5.9E-01 & 6.012 & 1.3E-01 & 0.812 & 4.0E-01 & 7.420 \\ 
4388   & 5.1E-01 & 2.274 & 1.9E-01 & 4.255 & 1.0E-01 & 0.997 & 2.8E-02 & 9.882 \\ 
7459   & 3.8E-01 & 1.124 & 1.2E-01 & 1.811 & 7.8E-02 & 0.991 & 4.0E-02 &  --   \\ 
12893  & 2.9E-01 & 0.930 & 8.2E-02 & 1.299 & 5.7E-02 & 1.182 & 9.2E-03 & 5.373 \\ 
22815  & 2.2E-01 & 1.063 & 5.6E-02 & 1.352 & 4.5E-02 & 0.825 & 9.7E-03 &  --   \\ 
39676  & 1.6E-01 & 1.018 & 4.2E-02 & 1.011 & 3.2E-02 & 1.258 & 5.4E-03 & 2.113 \\ 
69933  & 1.2E-01 & 1.027 & 3.1E-02 & 1.070 & 2.5E-02 & 0.872 & 4.0E-03 & 1.013 \\ 
121728 & 9.2E-02 & 1.019 & 2.3E-02 & 1.041 & 1.8E-02 & 1.152 & 3.0E-03 & 1.106 \\ 
213431 & 7.0E-02 & 0.966 & 1.8E-02 & 1.024 & 1.4E-02 & 0.903 & 2.3E-03 & 0.923 \\ 
374444 & 5.3E-02 & 1.020 & 1.3E-02 & 1.015 & 1.0E-02 & 1.094 & 1.7E-03 & 1.042 \\ 
\hline 
\end{tabular}
		
\medskip
		
\begin{tabular}{c|c||c|c||c|c||c}
\hline
$\re(\lambda)$ & $\sr(\lambda)$ & $\re(\vec{\bsi})$ & $\sr(\vec{\bsi})$ & $\Theta_{\tBFD}$ & $\eff(\Theta_{\tBFD})$ &  iter \\  \hline \hline
1.2E-01 &   --  & 3.2E-00 &   --  & 2.5E+01 & 0.129 & 5 \\ 
3.8E-00 &   --  & 5.5E-00 &   --  & 1.1E+01 & 0.525 & 4 \\ 
5.9E-01 & 7.406 & 1.3E-00 & 5.714 & 5.4E-00 & 0.245 & 4 \\ 
4.9E-02 & 9.271 & 5.5E-01 & 3.247 & 3.6E-00 & 0.156 & 5 \\ 
6.7E-02 &   --  & 4.1E-01 & 1.139 & 2.7E-00 & 0.153 & 5 \\ 
1.2E-02 & 6.168 & 3.1E-01 & 1.030 & 2.0E-00 & 0.152 & 5 \\      
1.7E-02 &   --  & 2.3E-01 & 1.065 & 1.5E-00 & 0.150 & 5 \\ 
5.1E-03 & 4.421 & 1.7E-01 & 1.037 & 1.1E-00 & 0.151 & 5 \\ 
2.5E-03 & 2.523 & 1.3E-01 & 1.025 & 8.5E-01 & 0.150 & 5 \\ 
2.0E-03 & 0.865 & 9.6E-02 & 1.025 & 6.4E-01 & 0.150 & 5 \\ 
1.4E-03 & 1.313 & 7.3E-02 & 0.967 & 4.9E-01 & 0.150 & 5 \\ 
7.0E-04 & 2.352 & 5.5E-02 & 1.023 & 3.7E-01 & 0.150 & 5 \\ 
\hline 
\end{tabular}
\caption{[{\sc Example 2}] $\bBR - \bRT_0 - \rP_0 - \rP_0$ primal-mixed scheme with adaptive refinement via $\Theta_{\tBFD}$.}\label{table:ex2-adaptive}
}
\end{center}
\end{table}

\begin{figure}[t]
\begin{center}
\includegraphics[width=8.4cm]{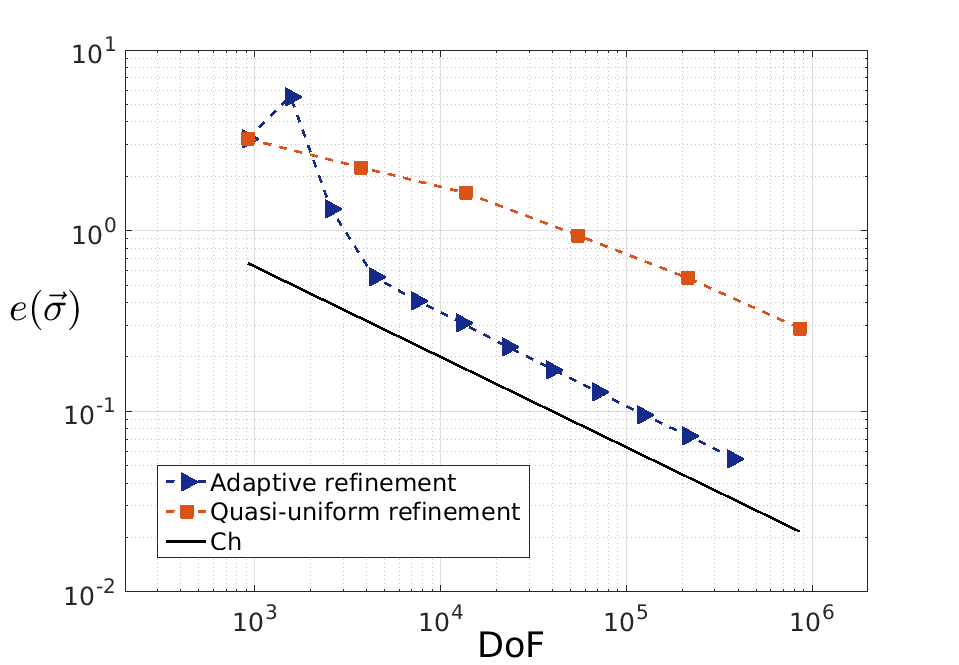}
		
\caption{[{\sc Example 2}] Log-log plot of $\re(\vec{\bsi})$ vs. ${\DOF}$ for quasi-uniform/adaptive refinements.}\label{fig:Ex2-loglog-errors}
\end{center}
\end{figure}

\begin{figure}[ht!]
\begin{center}
\includegraphics[width=5.41cm]{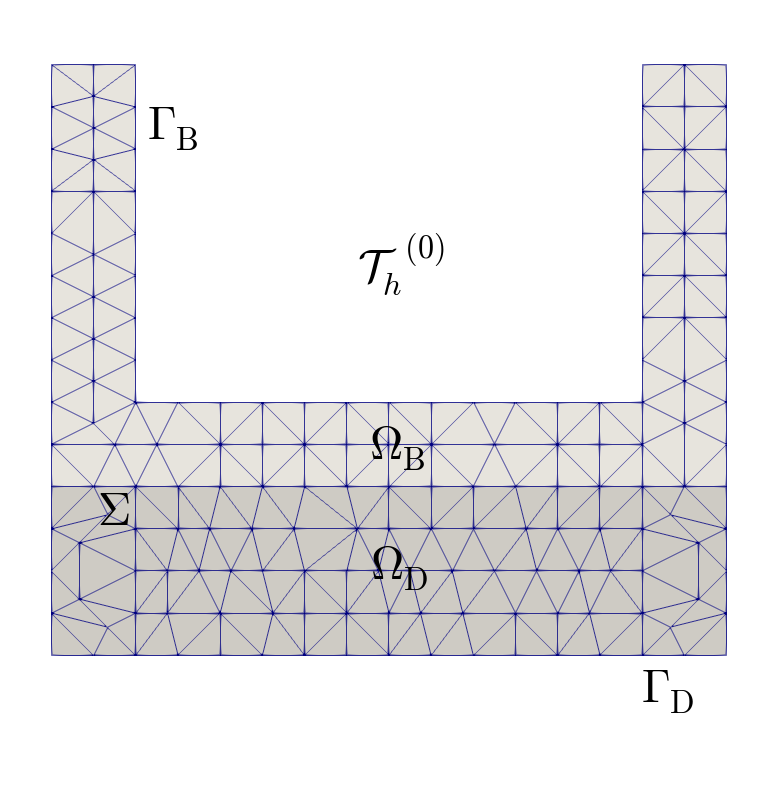}\quad
\includegraphics[width=5.41cm]{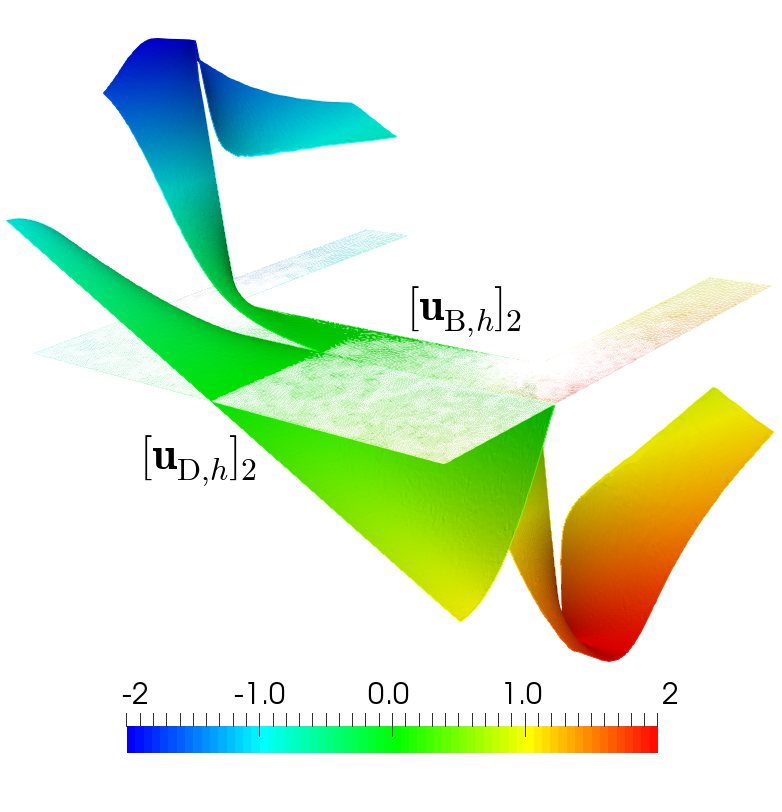}\quad	
\includegraphics[width=5.41cm]{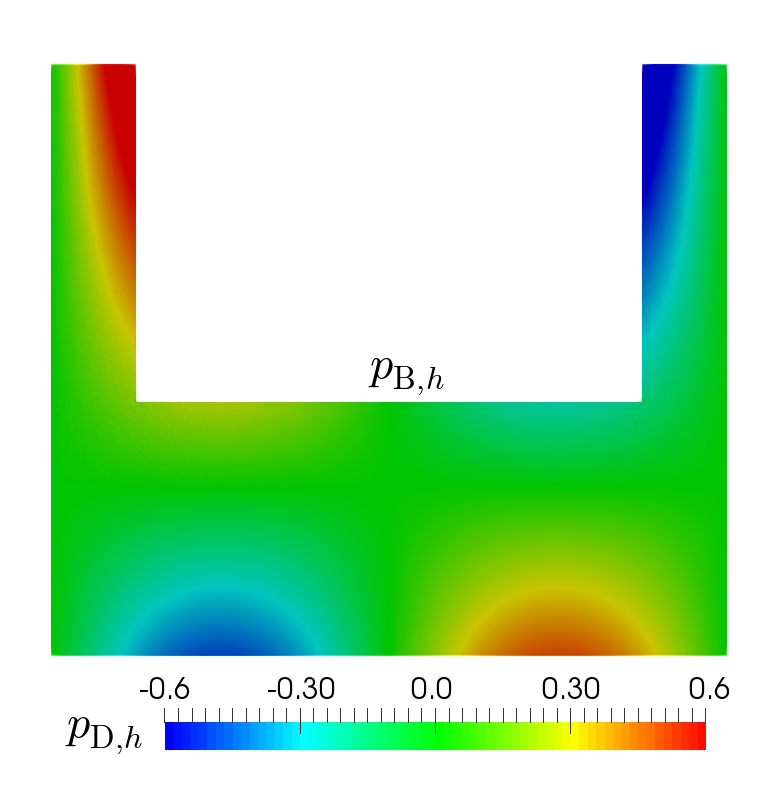}	
		
\vspace{-0.1cm}
		
\caption{[{\sc Example 2}] Initial mesh, second velocity component and pressure field in the whole domain.}\label{figure:Ex2-domain-solutions}
\end{center}
\end{figure}

\begin{figure}[t]
\begin{center}
\includegraphics[width=5.55cm]{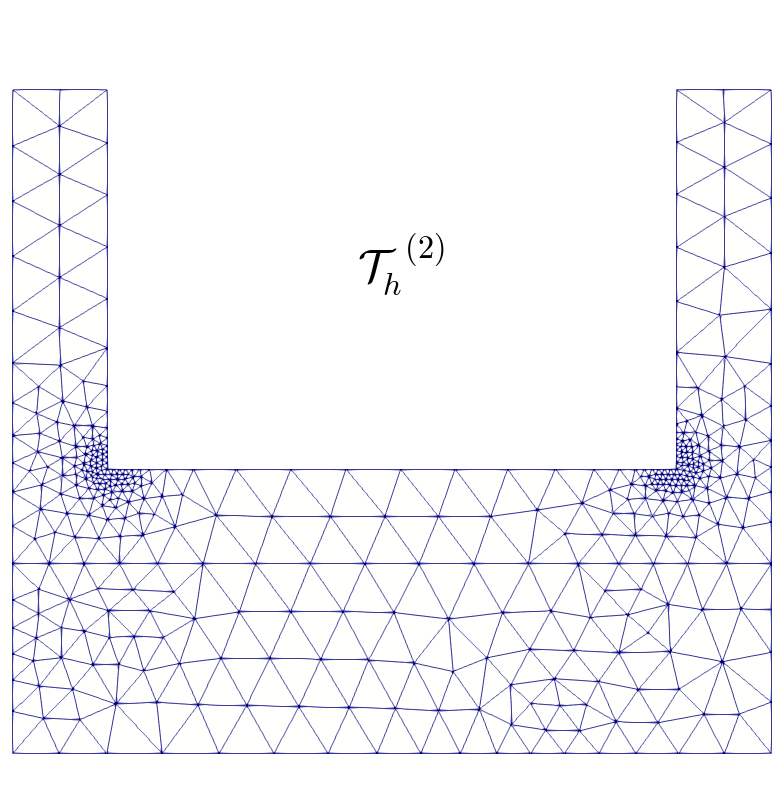}
\includegraphics[width=5.55cm]{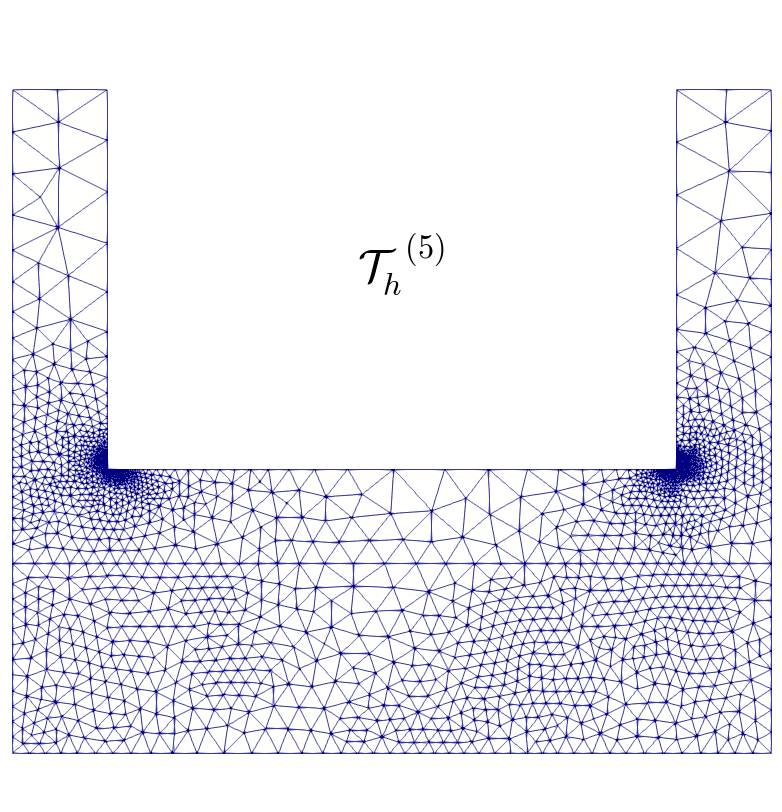}
\includegraphics[width=5.55cm]{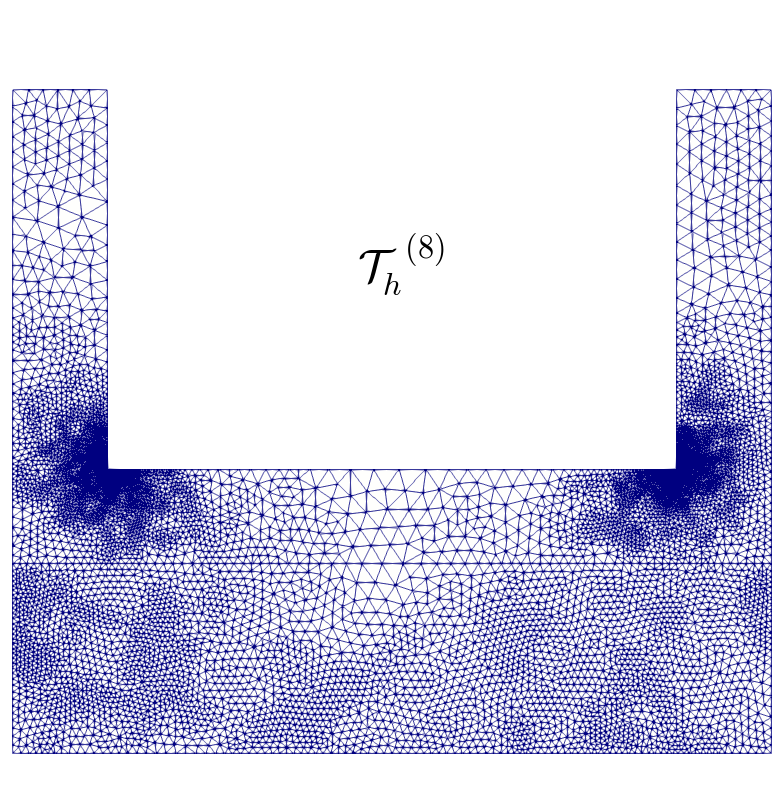}
		
\vspace{-0.1cm}
\caption{[{\sc Example 2}] Three snapshots of adapted meshes according to the indicator $\Theta_{\tBFD}$.}\label{figure:Ex2-adapted-meshes}
\end{center}
\end{figure}

\subsection*{Example 3: Adaptivity for flow through a heterogeneous porous media.}

Inspired by \cite[Example 2 in Section 5]{cd2023}, we finally focus on studying the behavior of the coupled Brinkman--Forchheimer/Darcy equations for modeling fluid flow through a heterogeneous porous medium, considering $\rho=4$ to incorporate the higher-order inertial correction $\tF\,|\bu_\rB|^2\,\bu_\rB$.
The domain is defined by the square $\Omega=\Omega_{\rB}\cup \Sigma\cup \Omega_{\rD}$,
where 
\begin{equation*}
\Omega_\rB :=(0,2)\times (0,1),\quad \Sigma:=(0,2)\times \{0\},\qan \Omega_\rD:=(0,2)\times (-1,0)\,,
\end{equation*}
with boundaries $\Gamma_\rB = \Gamma_{\rB,\textrm{left}}\cup \Gamma_{\rB,\textrm{top}}\cup \Gamma_{\rB,\textrm{right}}$ and $\Gamma_\rD=\Gamma_{\rD,\textrm{left}}\cup \Gamma_{\rD,\textrm{bottom}}\cup \Gamma_{\rD,\textrm{right}}$, respectively (see the first plot of Figure \ref{figure:Ex3-domain-solutions} below). 
The problem parameters are $\mu=1$, $\tF=10^4$, $\bK_{\rB}=10^{-1}\times\bbI$ and $\bK_{\rD}=10^{-3}\times\bbI$.
In turn, the right-hand side data $\f_\rB, \f_{\rD}$, and $g_\rD$ are chosen as zero, and the boundary conditions are
\begin{equation*}
\begin{array}{c}
\ds \bu_\rB = (-10\,x_2\,(x_2 - 1), 0)^\rt  \qon \Gamma_{\rB,\textrm{left}}\,,\quad
\bu_{\rB} = \0 \qon \Gamma_{\rB,\textrm{top}} \,,\quad
\bsi_\rB\bn = \0 \qon \Gamma_{\rB,\textrm{right}} \,, \\[2ex]
\ds p_\rD = 0 \qon \Gamma_{\rD,\textrm{bottom}} \,,\quad
\bu_\rD\cdot\bn = 0 \qon \Gamma_{\rD,\textrm{left}}\cup \Gamma_{\rD,\textrm{right}} \,.
\end{array}
\end{equation*}
In Table \ref{table:ex3-adaptive}, we compute the rate of convergence of the global estimator $\Theta_{\tBFD}$ for eight adapted meshes. Although the effectivity index cannot be computed for this example, the results demonstrate optimal convergence, confirming the reliability and efficiency of the estimator, as well as the accuracy of the obtained approximation.
In Figure \ref{figure:Ex3-domain-solutions}, we plot the second component of the velocity and the magnitude of the velocity over the entire domain, computed using the $\bBR - \bRT_0 - \rP_0 - \rP_0$ primal-mixed scheme on a mesh with $57,006$ triangular elements (corresponding to $196,733\,\DOF$) obtained via $\Theta_{\tBFD}$. As expected, most of the flow moves from left to right through the more permeable Brinkman--Forchheimer domain, while part of it is diverted into the less permeable Darcy region due to the zero pressure at the bottom of the domain. The continuity of the normal velocity across the interface is preserved, illustrating mass conservation on $\Sigma$.
Additionally, Figure \ref{figure:Ex3-adapted-meshes} shows snapshots of some adapted meshes generated using $\Theta_{\tBFD}$. We observe appropriate refinement in regions with higher velocity, indicating that the estimator $\Theta_{\tBFD}$ effectively captures challenging model parameters and localizes areas where the solutions exhibit greater complexity. This ensures computational resources are efficiently utilized where they are most needed.

\begin{table}[ht]
\begin{center}   
\small{		
\begin{tabular}{c||c|c|c|c|c|c|c|c} \hline
${\DOF}$       & 1807 & 4207 & 10080 & 17800 & 31928 & 57275 & 105724 & 196733 \\ 
${\tt it}$     & 8 & 8 & 8 & 8 & 8 & 8 & 8 & 8 \\ \hline\hline
$\Theta_\tBFD$ & 6.0E+03 & 2.2E+03 & 1.4E+03 & 1.0E+03 & 7.9E+02 & 6.0E+02 & 4.5E+02 & 3.4E+02 \\ 
$\sr(\Theta_\tBFD)$ & -- & 2.353 & 1.125 & 0.955 & 0.906 & 0.932 & 0.923 & 0.978 \\ \hline
\end{tabular}
}
\caption{[{\sc Example 3}] Number of degrees of freedom, Newton iteration count, global estimator, and rate of convergence of the global estimator.}\label{table:ex3-adaptive}	
\end{center}
\end{table}

\begin{figure}[ht!]
\begin{center}
\includegraphics[width=6.3cm]{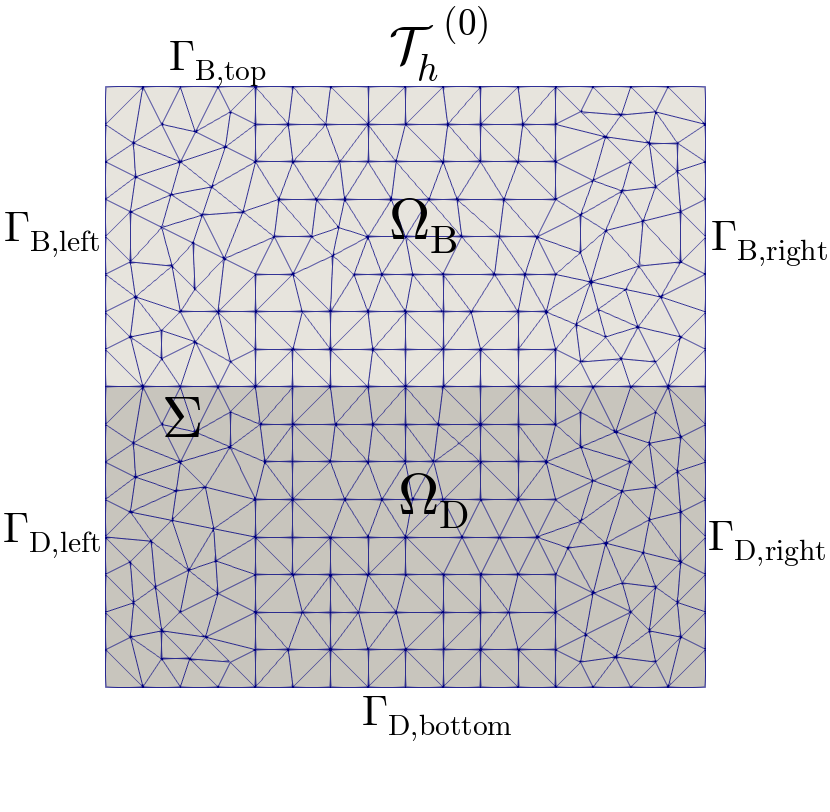}\quad
\includegraphics[width=4.9cm]{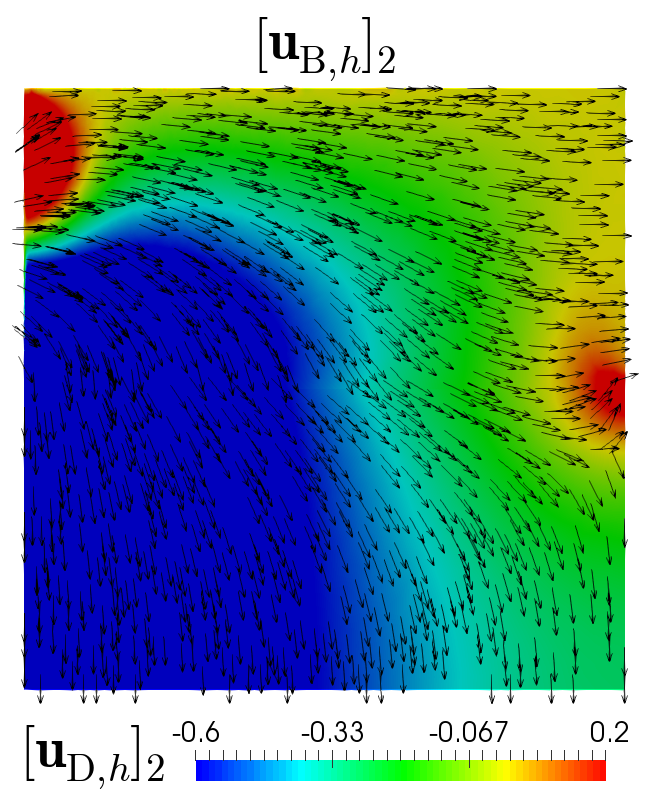}\quad	
\includegraphics[width=4.9cm]{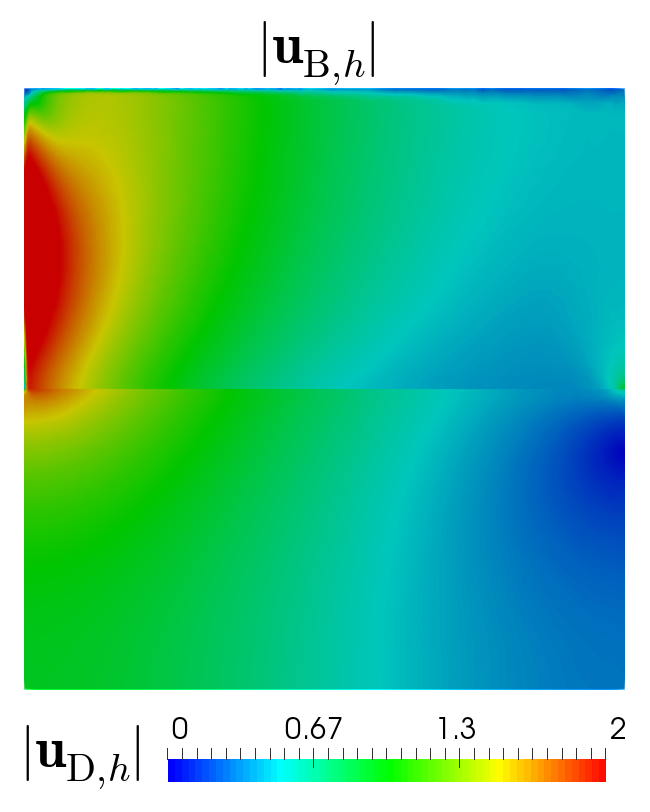}
		
\caption{[{\sc Example 3}] Initial mesh, second velocity component and magnitude of the velocity in the whole domain.}\label{figure:Ex3-domain-solutions}
\end{center}
\end{figure}

\begin{figure}[ht!]
\begin{center}
\includegraphics[width=5.5cm]{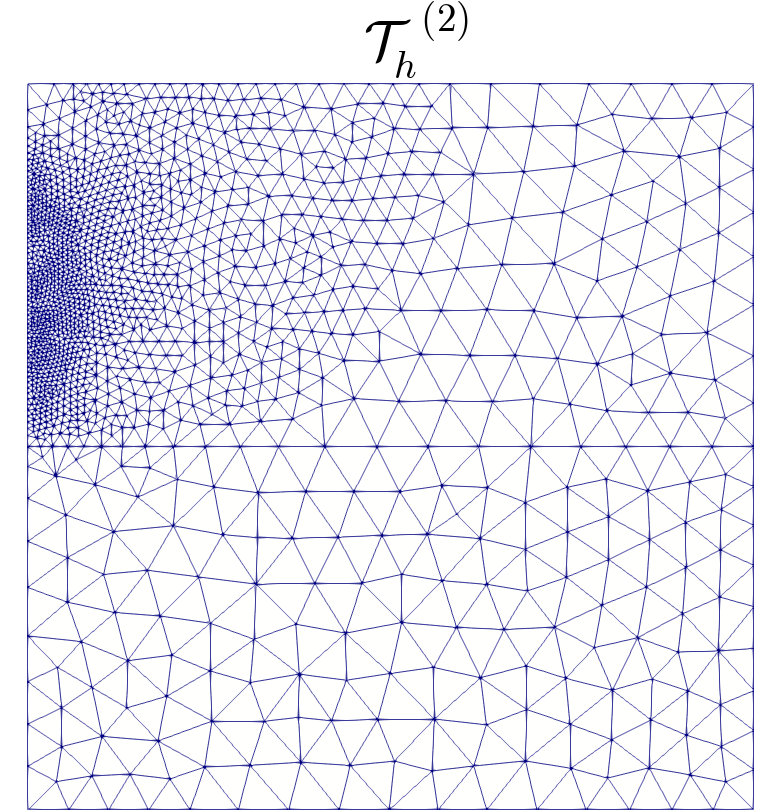}
\includegraphics[width=5.5cm]{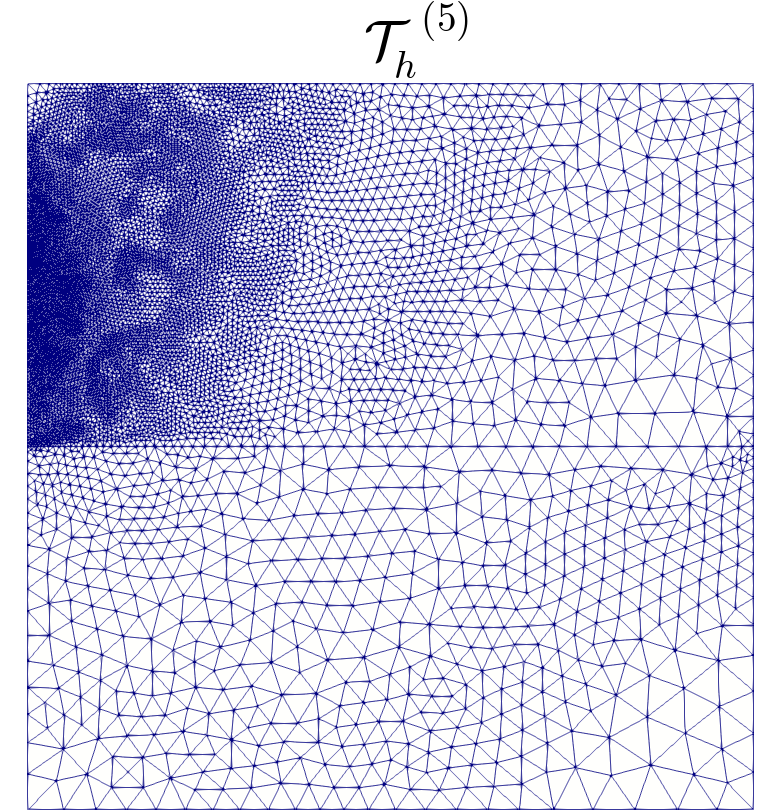}
\includegraphics[width=5.5cm]{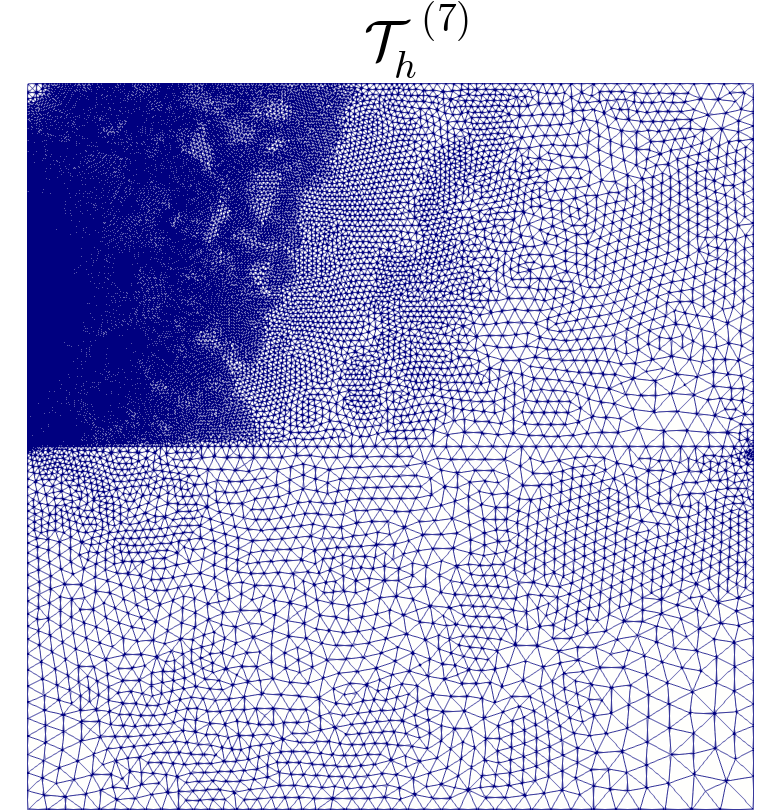}
		
\caption{[{\sc Example 3}] Three snapshots of adapted meshes according to the indicator $\Theta_{\tBFD}$.}\label{figure:Ex3-adapted-meshes}
\end{center}
\end{figure}

\appendix

\section{Preliminaries for reliability}\label{app:preliminaries-for-reliability}

We introduce a few useful notations for describing local information on elements and edges. 
First, given $T\in\cT^\rB_h\cup\cT^\rD_h$, we let $\cE(T)$ be the set of edges of $T$, 
and denote by $\cE_h$ the set of all edges of $\cT^\rB_h\cup\cT^\rD_h$, subdivided as follows:
\begin{equation*}
\cE_h = \cE_h(\Gamma_\rB)\cup \cE_h(\Gamma_\rD)\cup \cE_h(\Omega_\rB)\cup \cE_h(\Omega_\rD)\cup \cE_h(\Sigma) \,, 
\end{equation*}
where $\cE_h(\Gamma_\star):=\big\{ e\in\cE_h: e\subseteq \Gamma_\star \big\}$, 
$\cE_h(\Omega_\star):=\big\{ e\in\cE_h: e\subseteq \Omega_\star \big\}$, for $\star\in\{\rB,\rD\}$, 
and the edges of $\cE_h(\Sigma)$ are exactly those forming the previously defined 
partition $\Sigma_h$, that is $\cE_h(\Sigma):=\big\{ e\in\cE_h:\,\, e\subseteq \Sigma \big\}$.
Moreover, $h_e$ stands for the length of a given edge $e$.
Also for each edge $e\in\cE_h$ we fix a unit normal vector $\bn_e:=(n_1,n_2)^\rt$, and 
let $\bt_e:=(-n_2,n_1)^\rt$ be the corresponding fixed unit tangential vector along $e$.
Now, let $\bv\in\bL^2(\Omega_\star)$ such that $\bv|_T\in\bC(T)$ on each $T\in\cT^\star_h$.
Then, given $e\in\cE(T)\cap\cE_h(\Omega_\star)$, we denote by $\jump{\bv\cdot\bt_e}$ 
the tangential jump of $\bv$ across $e$, that is, 
$\jump{\bv\cdot\bt_e} := \big( (\bv|_T)|_e - (\bv|_{T'})|_e \big)\cdot\bt_e$, 
where $T$ and $T'$ are the triangles of $\cT^\star_h$ having $e$ as a common edge.
In addition, for $\btau\in\bbL^2(\Omega_\star)$ such that $\btau|_T\in\bbC(T)$, 
we let $\jump{\btau\,\bn_e}$ be the normal jump of $\btau$ across $e$, that is, 
\begin{equation}\label{eq:jump-tensor-n}
\jump{\btau\,\bn_e} := \big( (\btau|_T)|_e - (\btau|_{T'})|_e\big) \bn_e\,.
\end{equation}
From now on, when no confusion arises, we simply write $\bn$ and $\bt$ 
instead of $\bn_e$ and $\bt_e$, respectively.
Finally, given scalar and vector valued fields $\phi$ and $\bv = (v_1, v_2)^\rt$, 
respectively, we set
\begin{equation*}
\bcurl(\phi) := \left(
\frac{\partial \phi}{\partial x_2},\, 
-\frac{\partial \phi}{\partial x_1}
\right)^\rt \qan
\ds \rot(\bv):= \frac{\partial v_2}{\partial x_1} - \frac{\partial v_1}{\partial x_2} \,,
\end{equation*}
where the derivatives involved are taken in the distributional sense.

Let us now recall the main properties of the Raviart--Thomas interpolator 
(see for instance \cite{Gatica, Raviart-Thomas}) 
and the Cl\'ement operator onto the space of continuous piecewise 
linear functions \cite{clement1975,Verfurth}. 
We begin with the former, denoted $\Pi_h : \bH^1(\Omega_\rD) \to \bH_{h,\Gamma_{\rD}}(\Omega_\rD)$, 
which is characterized by the identity
\begin{equation}\label{eq:RT-interpolation}
\int_e \Pi_h(\bv)\cdot\bn = \int_e \bv\cdot\bn \quad \forall\,\text{ edge  $e$ of } \cT_h^\star. 
\end{equation} 
Moreover, as a straightforward consequence of \eqref{eq:RT-interpolation}, there holds:  
\begin{equation*}
\div(\Pi_h (\bv)) \,=\, \cP_h^\rD(\div(\bv)) \,,
\end{equation*}
where $\cP_h^\rD$ is the $\L^2(\Omega_\rD)$-orthogonal projector onto 
the piecewise constant functions on $\Omega_\rD$. 
The local approximation properties of $\Pi_h$ are established in the following lemma. 
For the corresponding proof we refer the reader to \cite{Gatica} (see also \cite{Brezzi-Fortin}).  
\begin{lem}\label{lem:Phih-properties}
There exist constants $c_1,c_2>0$, independent of $h$, such that 
for all $\bv \in \bH^1(\Omega_\rD)$ there holds
\begin{equation}\label{eq:property-Pi-1}
\|\bv - \Pi_h(\bv)\|_{0,T} \,\leq\, c_1\,h_T\,\|\bv\|_{1,T} \quad \forall\,T \in \cT_h^\rD
\end{equation}
and
\begin{equation}\label{eq:property-Pi-2}
\|\bv\cdot\bn - \Pi_h(\bv)\cdot\bn\|_{0,e} \,\leq\, c_2\,h_e^{1/2}\,\|\bv\|_{1,T_e} \quad \forall\,e \in \cE_h,
\end{equation}
where $T_e$ is a triangle of $\cT_h^\rD$ containing the edge $e$ on its boundary.
\end{lem}

In turn, we make use of the Cl\'ement interpolation operator 
$\cI_h^\star:\H^{1}(\Omega_\star)\to \X_h(\Omega_\star)$, with $\star \in \{\rB,\rD\}$, where
\begin{equation}\label{eq:Xh-Clement-space}
\X_h(\Omega_\star) \,:=\, \Big\{v \in \cC(\overline{\Omega}_\star) :\quad v|_T \in \rP_1(T) \quad \forall\,T \in \cT_h^\star\Big\} \,.
\end{equation}
The local approximation properties of this operator are established in 
the following lemma (see \cite[Lemma 3.1]{Verfurth} for details):
\begin{lem}\label{lem:clement}
For each $\star \in \{\rS,\rD\}$ there exist constants $c_3,c_4>0$, independent of $h_\star$, 
such that for all $v \in \H^{1}(\Omega_\star)$ there hold
\begin{equation*}
\|v - \cI^\star_h(v)\|_{0,T} \,\leq\, c_3\, h_T\,\|v\|_{1,\Delta_\star(T)} \quad \forall\, T\in\cT^\star_h,
\end{equation*}
and
\begin{equation*}
\|v - \cI^\star_h(v)\|_{0,e} \,\leq\, c_4\, h^{1/2}_e\, \|v\|_{1,\Delta_\star(e)} \quad \forall\, e\in\cE_h,
\end{equation*}
where
\begin{equation*}
\Delta_\star(T) \,:=\, \cup \Big\{T'\in\cT^\star_h: \,\, T'\cap T\neq \emptyset \Big\} \qan
\Delta_\star(e) \,:=\, \cup \Big\{T'\in\cT^\star_h: \,\, T'\cap e \neq \emptyset \Big\} \,.
\end{equation*}
\end{lem}
In particular, a vector version of $\cI^\rB_h$, say $\bcI^\rB_h:\bH^1(\Omega_\rB)\to \bX_h(\Omega_\rB)$, 
which is defined component-wise by $\cI^\rB_h$, will be needed as well.

Finally, we present a stable Helmholtz decomposition for $\bH_{\Gamma_\rD}(\div;\Omega_\rD)$. 
Indeed, as noted in \cite[Lemma 3.9]{agr2016}, this decomposition requires the boundary $\Gamma_\rD$ to lie within a ``{\it convex part}'' of $\Omega_\rD$. 
Specifically, $\Gamma_\rD$ must be contained in the boundary of a convex domain that fully encompasses $\Omega_\rD$. We begin by introducing the following subspace of $\H^1(\Omega_\rD)$,
\begin{equation}\label{eq:H1-GammaD-space}
\H^1_{\Gamma_\rD} (\Omega_\rD) \,:=\, \Big\{ \eta_\rD\in \H^1(\Omega_\rD) :\quad \eta_\rD = 0 \qon \Gamma_\rD \Big\},
\end{equation}
and proceed to establish a suitable Helmholtz decomposition for the space $\bH_{\Gamma_\rD}(\div;\Omega_\rD)$.
\begin{lem}\label{lem:Helmholtz-decomposition}
Assume that $\Omega_\rD$ is a simply connected
domain and that $\Gamma_\rD$ is contained in the boundary of a convex part of $\Omega_\rD$, that is there exists a convex domain $\Xi$ such that $\ov{\Omega}_\rD\subseteq \Xi$ and $\Gamma_\rD\subseteq\partial\Xi$. 
Then, for each $\bv_\rD\in\bH_{\Gamma_\rD}(\div;\Omega_\rD)$, 
there exist $\bw_\rD\in\bH^1(\Omega_\rD)$ and $\beta_\rD\in\H^1_{\Gamma_\rD}(\Omega_\rD)$ such that
\begin{equation*}
\bv_\rD \,=\, \bw_\rD + \bcurl(\beta_\rD) \qin \Omega_\rD
\qan
\|\bw_\rD\|_{1,\Omega_\rD} + \|\beta_\rD\|_{1,\Omega_\rD} 
\,\leq\, \Chel\,\|\bv_\rD\|_{\div;\Omega_\rD}\,,
\end{equation*}
where $\Chel$ is a positive constant independent of all the foregoing variables.
\end{lem}

\section{Preliminaries for efficiency}\label{app:preliminaries-for-efficiency}

For the efficiency analysis of $\Theta_{\tBFD}$ (cf. \eqref{eq:global-estimator}), we proceed as in  
\cite{cgos2021}, \cite{bg2010} and \cite{cgo2023}, and apply the localization technique based on bubble 
functions, along with inverse and discrete trace inequalities. 
For the former, given $T\in \cT_h:=\cT^\rB_h\cup \cT^\rD_h$ and $e\in \cE(T)$, 
we let $\phi_T$ and $\phi_e$ be the usual element-bubble and edge-bubble functions (cf. \cite[eqs. (1.5) and (1.6)]{Verfurth}), which respectively satisfy
\begin{equation}\label{eq:buble-property-null-boundary-T}
\phi_T \in \rP_3(T), \quad \text{supp}(\phi_T) \subseteq T, \quad \phi_T = 0 \qon \partial T 
\qan 0 \leq \phi_T \leq 1 \qin T\quad\mbox{and}
\end{equation}
\begin{equation}\label{eq:buble-property-null-boundary-e}
\phi_e|_T \in \rP_2(T), \quad \text{supp}(\phi_e) \subseteq \omega_e, \quad \phi_e = 0 \qon \partial T\setminus e 
\qan 0 \leq \phi_e \leq 1 \qin \omega_e\,,
\end{equation}
with $\omega_e:=\cup\big\{ T'\in \cT_h:\,e\in \cE(T') \big\}$.
We also recall from \cite{Verfurth} that, given $k\in\bbN\cup\{0\}$, there exists an extension operator $L:C(e)\to C(\omega_e)$ that satisfies $L(p)\in P_k(T)$ and $L(p)|_e=p$ $\forall p\in P_k(e)$. A corresponding vector version of $L$, that is the component-wise application of $L$, is denoted by $\bL$.
The specific properties of $\phi_T, \phi_e$, and $L$ are collected in the following lemma, for whose proof we refer to \cite[ Lemma 3.3]{Verfurth}.
\begin{lem}\label{lem:properties-bubble}
Given $k\in\bbN\cup\{0\}$, there exist positive constants $c_1,~c_2$, $c_3$ and $c_4$, depending only on $k$ and the shape regularity of the triangulations (minimum angle condition), such that for each triangle $T$ and $e\in\cE(T)$, there hold
\begin{eqnarray}
\|\phi_T q\|^2_{0,T}\leq\|q\|^2_{0,T}\leq c_1\|\phi^{1/2}_T q\|^2_{0,T}\qquad \forall q\in P_k(T),\label{eq:buble-property-1} \\[2ex]
\|\phi_e L(p)\|^2_{0,e}\leq\|p\|^2_{0,e}\leq c_2\|\phi^{1/2}_e p\|^2_{0,e}\qquad \forall p\in P_k(e),\label{eq:buble-property-2}
\end{eqnarray}
and
\begin{equation}\label{eq:buble-property-3}
c_3h_e^{1/2}\|p\|_{0,e}\le  \|\phi^{1/2}_e L(p)\|_{0,T}\leq c_4h_e^{1/2}\|p\|_{0,e}\qquad \forall p\in P_k(e).
\end{equation}
\end{lem}

In turn, the aforementioned inverse inequality is stated as follows (cf. \cite[Lemma 1.138]{Girault-Raviart}).
\begin{lem}\label{lem:inverse-inequality}
Let $k$, $\ell$, and $m$ be non-negative integers such that $m\leq \ell$, and let $T\in \cT_h$. Then, there exists $c > 0$, independent of $h$ and $T$, but depending on $k$, $\ell$, $m$, and the shape regularity of the triangulations, such that
\begin{equation}\label{eq:inverse-inequality}
\|v\|_{\ell,T} \,\leq\, c\,h^{m-\ell}_T \|v\|_{m,T}\quad \forall \, v\in \rP_{k}(T) \,.
\end{equation}
\end{lem}

Finally, we also recall from \cite[Theorem 3.10]{Agmon}
the discrete trace inequality.
\begin{lem}
There exits $c>0$, depending only on the shape regularity of the triangulations, such that for each $T\in \cT_h$ and $e\in \cE(T)$, there holds
\begin{equation}\label{eq:discrete-trace-inequality}
\| v\|_{0,e}^{2} \,\leq\, c \,\Big\{ h_{T}^{-1}\, \| v\|_{0,T}^{2} 
+ h_{T}\,|v|_{1,T}^{2} \Big\} \quad \forall \, v\in \H^{1}(T)\,.
\end{equation}
\end{lem}


\end{document}